\documentclass[10pt,a4paper,reqno]{amsart} 
\usepackage{latexsym}
\usepackage{verbatim}
\usepackage{epsfig}
\usepackage{rotating}
\usepackage{amssymb}
\usepackage[T1]{fontenc}
\usepackage{afterpage}
\usepackage{color}
\usepackage{dsfont}
\usepackage{url}
\usepackage[utf8]{inputenc}
\usepackage{amssymb,amsfonts,amsmath,stmaryrd,bbm}
\usepackage[plainpages=false,pdfpagelabels,colorlinks=true,citecolor=blue,hypertexnames=false]{hyperref}

\catcode`\@=11
\def\section{\@startsection{section}{1}%
 \z@{.7\linespacing\@plus\linespacing}{.5\linespacing}%
 {\normalfont\bfseries\scshape\centering}}

\def\subsection{\@startsection{subsection}{2}%
  \z@{.5\linespacing\@plus\linespacing}{.5\linespacing}%
  {\normalfont\bfseries\scshape}}

\def\subsubsection{\@startsection{subsubsection}{3}%
 \z@{.5\linespacing\@plus\linespacing}{-.5em}
 {\normalfont\bfseries}}
\catcode`\@=12

\newtheorem{Theorem}{Theorem}[section]

\newtheorem{Proposition}[Theorem]{Proposition}
\newtheorem{Definition}[Theorem]{Definition}
\newtheorem{Corollary}[Theorem]{Corollary}

\def\qed{$\hfill{\vrule height 3pt width 5pt depth 2pt}$}

\def\qschoose{\atopwithdelims\llbracket \rrbracket}
\newcommand\bi[2]{{{#1}\atopwithdelims(){#2}}}

\newfont{\bbold}{msbm10 scaled \magstep1}
\newfont{\bbolds}{msbm7 scaled \magstep1}

\newcommand{\ns}{\mathbb{N}}

\newcommand{\qs}{\mathbb{Q}}

\newcommand{\HH}{H}
\newcommand{\JJ}{J}
\newcommand{\KK}{K}
\newcommand{\JI}{\mathcal J}
\newcommand{\KI}{\mathcal K}
\newcommand{\HI}{{\JI}_e}

\newcommand{\llp}{(\!(}
\newcommand{\rrp}{)\!)}

\newcommand{\ctA}{\tilde{\mathcal A}}
\newcommand{\ctB}{\tilde{\mathcal B}}
\newcommand{\ctC}{\tilde{\mathcal C}}
\newcommand{\ctD}{\tilde{\mathcal D}}
\newcommand{\cA}{\mathcal A}
\newcommand{\cB}{\mathcal B}
\newcommand{\cC}{\mathcal C}
\newcommand{\cD}{\mathcal D}

\newcommand{\cM}{\mathcal M}
\newcommand{\cN}{\mathcal N}
\newcommand{\cP}{\mathcal P}
\newcommand{\cS}{\mathcal S}
\newcommand{\cT}{\mathcal T}
\newcommand{\cW}{\mathcal W}
\newcommand{\cR}{\mathcal R}
\newcommand{\cL}{\mathcal L}

\DeclareMathOperator{\id}{id}
\DeclareMathOperator{\Pol}{Pol}
\newcommand{\beq}{\begin{equation}}
\newcommand{\eeq}{\end{equation}}

\newcommand{\gf}{generating function}
\newcommand{\gfs}{generating functions}

\newcommand{\al}{\alpha}
\newcommand{\be}{\beta}

\newcommand{\Ba}{B^{(\rm a)}}
\newcommand{\cBa}{\cB^{(\rm a)}}
\newcommand{\tBa}{\hat B^{(\rm a)}}
\newcommand{\Bna}{B^{(\rm na)}}
\newcommand{\cBna}{\cB^{(\rm na)}}

\def\emm#1,{{\em #1}}

\newcommand{\vO}{\check{O}}

\newcommand{\vcO}{\check{\mathcal{O}}}
\newcommand{\tA}{\tilde A}
\newcommand{\tB}{\tilde B}
\newcommand{\tC}{\tilde C}
\newcommand{\tD}{\tilde D}
\newcommand{\AL}{L}
\newcommand{\Ac}{A^{(\rm c)}}

\def\qchoose{\atopwithdelims[]}
\newcommand{\mbmChange}[1]{{{#1}}}
\newcommand{\bjnChange}[1]{{{#1}}}
\newcommand{\mbmChangenew}[1]{{{#1}}}

\begin{document}
\title[Length enumeration of fully commutative elements]
{Length enumeration of fully commutative elements\\in 
 finite and affine Coxeter groups}

\author[R. Biagioli]{Riccardo Biagioli}
\address{R. Biagioli, F. Jouhet, and P. Nadeau: Univ Lyon, Universit\'e Claude Bernard Lyon 1, CNRS UMR 5208, Institut Camille Jordan, F-69622 Villeurbanne Cedex, France}
\email{biagioli, jouhet, nadeau@math.univ-lyon1.fr}

\author[M. Bousquet-M\'elou]{Mireille Bousquet-M\'elou}
\address{M. Bousquet-M\'elou: CNRS, LaBRI, Universit\'e de Bordeaux, 
351 cours de la Lib\'eration, 33405 Talence, France}
\email{mireille.bousquet@labri.fr}

\author[F. Jouhet]{Frédéric Jouhet}

\author[P. Nadeau]{Philippe Nadeau}

\begin{abstract} An element $w$ of a Coxeter group $W$ is said to be \emm
fully commutative, if any reduced expression of $w$
can be obtained from any other by transposing adjacent pairs of
generators. These elements were described in 1996 by Stembridge in the
case of finite irreducible groups, and  more recently by Biagioli, Jouhet and Nadeau (BJN)
in the affine cases. We {focus here on}  the length enumeration of these
elements. Using a recursive description, BJN established for the
associated \gfs\  systems of non-linear $q$-equations. Here, we show
that an alternative recursive description leads to  explicit
expressions for these \gfs.
\end{abstract}
\date{\today}

\maketitle


\section{Introduction}
Let $(W,S)$ be a Coxeter system. An element $w$ of $W$ is said to be \emm
fully commutative, (or fc for short) if any reduced expression of $w$
can be obtained from any other by transposing adjacent pairs of
generators.  {For instance, in the finite or affine symmetric group,
fc elements  coincide with 321-avoiding
permutations~\cite{BJS,Gre321}.}
The description and enumeration of fully commutative
elements has been of interest in the algebraic combinatorics
literature for about 20 years, starting with the work of
Stembridge~\cite{St1,St2,St3}. 
We refer to the above papers and
 to~\cite{BJN-long} for {motivations of this topic}.
Stembridge classified 
 Coxeter groups having finitely many fc elements, and was able to count those elements in each case~\cite{St1,St3}. 

More recently, several authors got interested, not only in the number
of such elements, but also in their $q$-enumeration, where the
variable $q$ records  their Coxeter
length~\cite{BJN-long,HanJon}. For instance, in the symmetric group
$A_2$, all elements except the maximal permutation are fc, and their
length enumeration yields the polynomial
$$
A_2^{FC}(q)=1+2q+2q^2.
$$
This point of view naturally 
extends the study to {arbitrary} Coxeter groups $W$ (having
{possibly} infinitely many fc elements), since
they still have finitely many elements of given length. In this case  the length
enumeration of fc elements in $W$ gives rise to a power series rather than a
polynomial. We denote this series by $W^{FC}(q)$.

 In particular, three of the authors of the present paper (BJN) were able to
characterize, for all families of  classical finite or affine Coxeter groups
($A_n$, $\tA_n$, $B_n$, etc.), the series $W_n^{FC}(q)$, and in fact,
the bivariate \gf\ 
\beq\label{Wxq}
W(x,q):=
\sum_{n} W_n^{FC}(q) x^n,
\eeq
{by} systems of non-linear
$q$-equations~\cite{BJN-long}. For instance, in the 
 $A$-case, the system
can be reduced to a single {quadratic} $q$-equation for a series denoted
$M^*(x)\equiv M^*(x,q)$:
$$
M^*(x)=1+xM^*(x)+xq(M^*(x)-1)M^*(xq).
$$
Using these systems of equations, 
{BJN}  were able to prove that
$W_n^{FC}(q)$ is always a rational function of $q$, with simple poles at
roots of unity: this means that the coefficients of these series are
ultimately periodic~\cite{BJN-long,JN-periods}. This was first proved in the $\tA$-case by
Hanusa and Jones~\cite{HanJon}. {Subsequently, one of us (N)}
showed  that $W^{FC}(q)$ is in fact rational for any
{Coxeter group}~$W$, and determined when its coefficients are ultimately periodic~\cite{nadeau}.

The aim of this paper is to provide closed form expressions for 
\gfs\ of the form~\eqref{Wxq} for all classical Coxeter groups.
Some of them turn out to be particularly
elegant. Here are, for instance, our results in the $A$- and $\tA$-cases.
 For $n\ge 0$, we denote 
\beq\label{q-fact}
(x)_n\equiv(x;q)_n= (1-x)(1-xq)\cdots(1-xq^{n-1}).
\eeq

\begin{Theorem}\label{thm:A-tA}
Let  $A(x,q)\equiv A$ and $\tA(x,q)\equiv \tA$ be the \gfs\ of fully
commutative elements of type $A$ and $\tA$, respectively defined by
$$
A=\sum_{n\ge 0} A_n^{FC}(q) x^n \qquad\hbox{and} \qquad \tA=
\sum_{n\ge 1} \tA_{n-1}^{FC}(q) x^n. 
$$
Then 
$$
A= \frac 1 {1-xq}\frac{\JJ(xq)}{\JJ(x)}
\qquad\hbox{and} \qquad \tA=- x \frac{\JJ'(x)}{\JJ(x)}- \sum_{n\ge 1}
 \frac{x^n  q^n}{1-q^n},
$$
where $\JJ(x)$ is the following series:
$$
\JJ(x)=\sum_{n\ge 0} \frac{(-x)^n q^{n\choose 2}}{(q)_n(xq)_n}.
$$
\end{Theorem}
The $A$-result was already obtained by Barcucci \emm et al., in terms
of pattern avoiding permutations~\cite{barcucci}. The $\tA$-result is
new, though another
(much more complicated) expression of this series was given by Hanusa and
Jones~\cite{HanJon}.

We obtain similar results for fc involutions, already considered by
Stembridge in 1998~\cite{St3}, and for which non-linear $q$-equations
were given in~\cite{BJN-inv}. We denote by $\cW^{FC}(q)$ (with a
calligraphic~$\cW$) the
length \gf \ of fc involutions in a Coxeter group $W$, and for a
family $W_n$ of such groups, we consider the \gf
$$
\cW(x,q)=
\sum_n  \cW_n^{FC}(q) x^n.
$$
For $n\ge 0$, we denote
 
\beq\label{double}
\llp x \rrp _n\equiv (x;q^2)_n=(1-x)(1-xq^2) \cdots (1-xq^{2n-2}).
\eeq

\begin{Theorem}\label{thm:A-tA-inv}
Let  $\cA(x,q)\equiv \cA$ and $\ctA(x,q)\equiv \ctA$ be the \gfs\ of fully
commutative involutions of type $A$ and $\tA$, respectively defined by
$$
\cA=\sum_{n\ge 0} \cA_n^{FC}(q) x^n \qquad\hbox{and} \qquad \ctA=
\sum_{n\ge 1} \ctA_{n-1}^{FC}(q) x^n. 
$$
Then
$$
\cA
=\frac{\JI(-xq)}{\JI(x)}
\qquad\hbox{and} \qquad 
\ctA= -x\,
\frac{ \JI'(x)}{\JI(x)},
$$
with
$$
\JI(x)=\sum_{n\geq0}\frac{(-1)^{\lceil n/2\rceil} x^nq^{n \choose
    2}}
{\llp q^2 \rrp_{\lfloor n/2\rfloor}}.
$$
\end{Theorem}

The story of this paper, and of the tools that we use, parallels the story
of the area enumeration of \emm convex polyominoes, (examples of such
objects appear in Figure~\ref{fig:fs}): early results on this
topic involved systems of non-linear $q$-equations arising from a
certain recursive description of these objects~\cite{bousquet-eq,delest-fedou,delest-fedou-diag-conv-dir}. These equations were solved
--- when they were solved --- by guessing and checking~\cite{delest-fedou,bousquet-fedou}. Then came a
method that could linearize some of them~\cite{prellberg-brak}. Finally, an alternative
recursive {description},
 pioneered as early as 1974 but overlooked for
two decades~\cite{klarner-rivest-conv}, led directly, in a constructive fashion, to
closed form
expressions~\cite{bousquet-vcd,owczarek-prellberg-sos}. This
{approach}  had been previously used to obtain \emm perimeter, (rather than area) \gfs\ of
polyominoes~\cite{temperley}. As we
shall see, fc elements of type 
$A$ are in fact very directly related to a  simple class of convex
polyominoes, called staircase (or: parallelogram) polyominoes.

\medskip
Here is now an outline of the paper. In Section~\ref{sec:heaps}, we
 recall Stembridge's description of fc elements in terms of \emm
heaps, (or partially commutative words), in the sense of
Viennot~\cite{viennot1}. 
Then the key point in the enumeration of
fc elements in families of irreducible finite or affine Coxeter groups
becomes the enumeration of the so-called \emm alternating heaps, over a path or a
cycle (Figure~\ref{fig:Alter}). 

In Section~\ref{sec:finite} we state our results for the finite
groups $A_n$, $B_n$ and $D_n$. The forms of the $B$- and $D$-series
are similar to those of the $A$-series shown in
{Theorems}~\ref{thm:A-tA} and~\ref{thm:A-tA-inv}. To obtain these results, the key
 problem is to count alternating heaps over a path \emm having
at most one piece in the rightmost column,. We solve this problem in
Section~\ref{sec:recursive} using a recursive approach, and this
yields the finite type results
stated in Section~\ref{sec:finite}.

In Section~\ref{sec:At} we address the $\tA$-case, which boils down to
counting alternating heaps over a cycle. There, our recursive
approach seems harder to implement. Instead, we start from the
$q$-equations of~\cite{BJN-long} describing  the $A$- and $\tA$-cases,
and use our results
for the $A$-case to solve them. We thus obtain the nice expressions of
{Theorems}~\ref{thm:A-tA} and~\ref{thm:A-tA-inv}. 

The other affine cases ($\tB, \tC$ and $\tD$) require to count
alternating heaps over a path, this time with no condition on the rightmost
column. The recursive approach of Section~\ref{sec:recursive} would
yield very heavy 
expressions. We use instead an alternative one, described in
Section~\ref{sec:column}: it yields simpler expressions, but they involve
positive \emm and negative, powers of $q$. This phenomenon has in fact
been witnessed in polyomino
enumeration already~(see~\cite{feretic2,feretic} and~\cite[Ex.~5.5.2]{goulden-jackson}).
We then derive from our results on alternating heaps the \gfs\ of fc
elements in the affine cases $\tB, \tC$ and $\tD$ (Section~\ref{sec:infinite}).

{We have checked all our expressions using the package {\sc
    GAP}. Details on this procedure are given after {Theorem}~\ref{thm:ABD}.}

\section{Fully commutative elements and heaps of generators}
\label{sec:heaps}

 Let $M$ be a square symmetric matrix indexed by a finite set $S$,
 satisfying $m_{ss}=1$ and, for $s\neq t$,
 $m_{st}=m_{ts}\in\{2,3,\ldots\}\cup\{\infty\}$. The {\em Coxeter
   group} $W$ associated with the  
 matrix $M$ is defined by
 {its set $S$ of} generators and by the following
 relations: all generators are reflections ($s^2=1$ for all $s\in S$),
and they satisfy 
 \emph{braid relations}:
\beq\label{braid}
\underbrace{sts\cdots}_{m_{st}}  = \underbrace{tst\cdots}_{m_{st}} \quad\text{if } m_{st}<\infty.
\eeq
 When $m_{st}=2$, the braid relation reduces to a  \emph{commutation
   relation} $st=ts$.  {Note that what we mean here by ``Coxeter group''
 is $W$ plus its presentation in terms of  $S$ and $M$. We refer
 to~\cite{bjorner-brenti-book,humphreys} for classical textbooks on this topic.}

The {\em Coxeter graph} 
 associated to $W$
 is the graph $\Gamma$ with vertex set $S$ and, for each pair $\{s,t\}$
with $m_{st}\geq 3$, an edge between $s$ and $t$. This edge is  labeled
by $m_{st}$ if $m_{st}>3$  and left unlabeled if $m_{st}=3$. Two
generators that are not joined by an edge commute. Note that
the graph completely characterizes the group. We will
consider here the main infinite families of finite and affine Coxeter
groups, whose 
Coxeter graphs are shown in Figure~\ref{fig:dynkin} 

\begin{figure}[!ht]
\includegraphics[width=\textwidth]{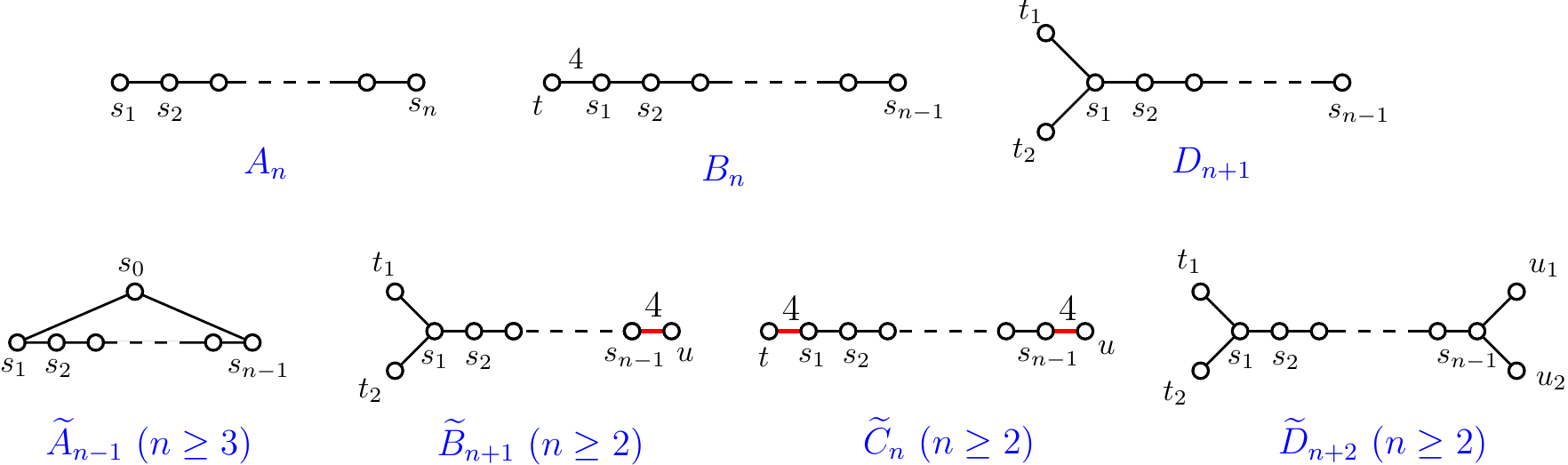}
\caption{Coxeter graphs for classical types.\label{fig:dynkin}}
\end{figure}

For $w\in W$, the {\em length} of $w$, denoted by $\ell(w)$, is the
minimum length $l$ of any expression $w=s_1\cdots s_l$ with $s_i\in
S$. The expressions of $w$ of length $\ell(w)$ are called
\emph{reduced}.
 A fundamental result in Coxeter group theory, sometimes called
the {\em Matsumoto property},  is that any reduced expression of $w$
 can be obtained from any other 
 using only the braid relations~\eqref{braid}. 

\begin{Definition}
An element $w\in W$ is \emph{fully commutative} (fc) if any reduced
expression 
of  $w$ can be obtained from any other 
 by using only commutation relations.
\end{Definition}

The set of words on the alphabet $S$, quotiented by the commutation
relations $st=ts$ for $s$ and $t$ such that $m_{st}=2$, forms a \emm partially commutative
monoid, in the sense of Cartier and Foata~ \cite{cartierfoata}. Its
elements, which are the commutation classes,  can be represented as special posets called, in Viennot's terminology~\cite{viennot1},
 \emm heaps over the Coxeter graph $\Gamma$,. For instance, the first
 heap of Figure~\ref{fig:Alter} represents the commutation class of the word
 $s_4s_3s_1s_2s_1s_2 s_5s_5s_5s_4s_3s_5s_2s_4s_6s_6s_1$. The
 correspondence between the word and the poset is rather intuitive,
 but let us still recall the
 definition   of these heaps~\cite[Def.~2.1]{viennot1}.

\begin{Definition} \label{defheap} 
A heap $(H,\prec,\epsilon)$ over the graph $\Gamma$ is a partially ordered set (poset) $(H,\prec)$ and
a labeling function $\epsilon:H\to S$ such that elements of $H$ labeled $s$
(resp. labeled $s$ or $t$) form a chain  for any $s\in S$
(resp. for any $s,t$ that are adjacent in $\Gamma$). \mbmChangenew{This
  chain is denoted $H_s$ (resp. $H_{st}$)}.
 Moreover these chains must
induce the ordering $\prec$ of~$H$. 

We consider heaps up to isomorphism, that is to say poset isomorphisms
which preserve the labeling functions.
\end{Definition}

We will call  \emph{points} the elements of $H$, each point being
labeled with a generator  $s$ of $S$. In our figures, all points with
label $s$ are placed on a vertical line above $s$ (Figure~\ref{fig:Alter}). We let $|H|_s$ denote the number of
points labeled $s$, and $|H|=\sum_{s\in S}|H|_s$  the total number
of points, also called \emm size, of $H$. Of course, this is the
length of the associated word on $S$.

In order to discuss  fc involutions, we will need to define {\em self-dual heaps}: a heap $(H,\prec,\epsilon)$ is
self-dual if it is isomorphic to $(H,\succ,\epsilon)$. {An
  example is shown on Figure~\ref{fig:A-inv}.}

Now take $w$ in the group $W$. Saying that $w$ is fc means
  precisely that all its reduced expressions correspond to the same
  heap $H$. We will often say that $H$ itself is fully commutative,
  and identify $w$  with the  heap $H$. 
 A characterization of fc heaps for an arbitrary Coxeter graph was
 given by  Stembridge~\cite[Prop.~3.3]{St1}. Underlying  this characterization is the notion of
 \emph{alternating heaps}.

\begin{Definition}\label{def:altheaps}
A heap $H$ over $\Gamma$ is \emm alternating, if, for all $s$ and $t$ that are adjacent in
$\Gamma$, the points of  the chain
$H_{st}$  are \emm alternatingly, labeled $s$ and $t$.
\end{Definition}

\begin{figure}[ht]
\begin{center}
\includegraphics[scale=0.7]{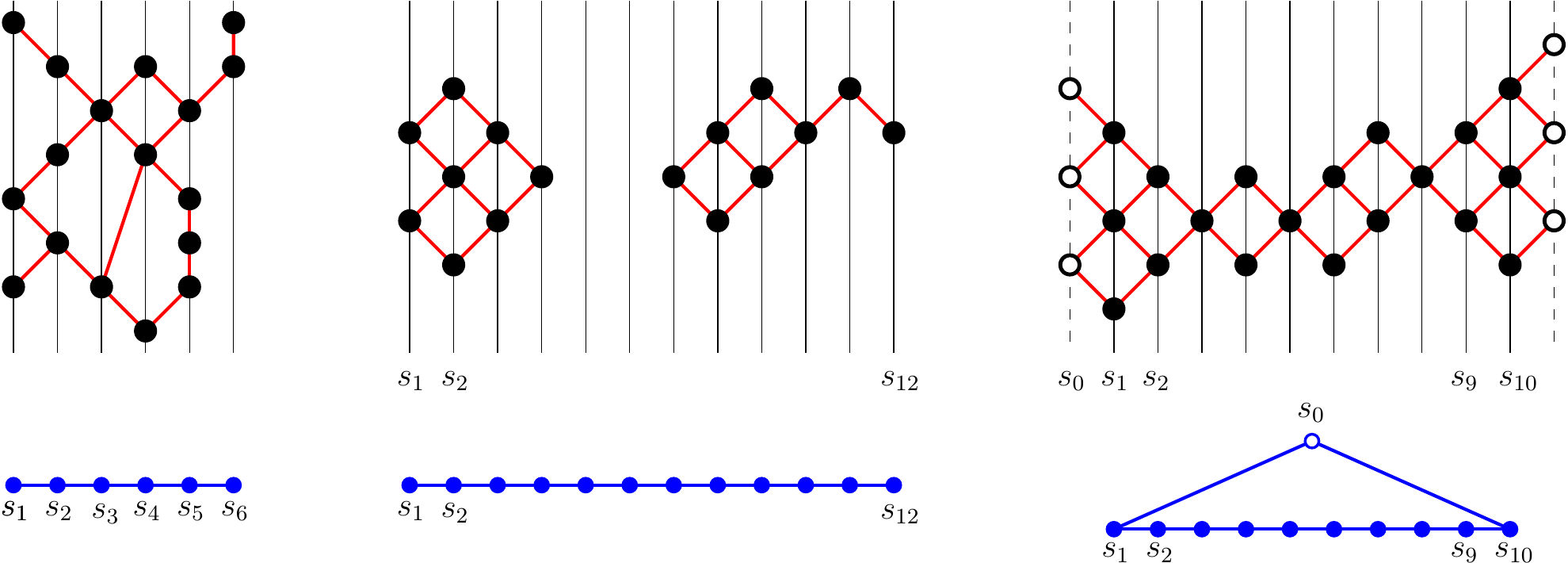}
\caption{Left: A heap over the 6-point path. Center: An alternating
  heap over the 12-point path. Right: an alternating heap over the 11-point cycle.} 
\label{fig:Alter}
\end{center}
\end{figure}

In this paper we  will consider alternating heaps over a path and
over a cycle (Figure~\ref{fig:Alter}, center and right). Then the chains
$H_{st}$  form increasing zigzag paths between the points of
neighboring columns. In the following sections, we do not draw these
paths, since they can be read off from the positions of points; see
for instance Figure~\ref{fig:A}.

\medskip
{We finish this section with some standard algebraic
  notation. For a ring $R$, we denote by $R[x]$ 
the ring of
polynomials 
in $x$ with coefficients in
$R$. If $R$ is a field, then $R(x)$ stands for the field of rational functions
in $x$. This notation is generalized to several variables in the usual
way.}

\section{Generating functions for finite types}
\label{sec:finite}
In this section, we state our results for fully commutative elements,
and then for fully commutative involutions, in Coxeter groups of finite
type. They will be proved in the next section.

\subsection{All fully commutative elements}
{Their} \gfs\ will be expressed in terms of two series $\JJ(x)$ and  $\KK(x)$,
defined~by
\beq\label{JK-def}
\JJ(x)=\sum_{n\ge 0} \frac{(-x)^n q^{n\choose 2}}{(q)_n(xq)_n}, \qquad\qquad
\KK(x)= \sum_{n\ge 0} \frac{x^n q^{n+1\choose 2} }{(xq)_{n}}\sum_{k=0}^{n}\frac{(-1)^k}{(q)_k},
\eeq  
{where the notation $(x)_n$ is defined by~\eqref{q-fact}.}

\begin{Theorem}\label{thm:ABD}
  Let $A(x,q)\equiv A$, $B(x,q)\equiv B$, and $D(x,q)\equiv D $ be the
  \gfs\  of fully commutative elements of types $A$, $B$ and $D$, defined
  respectively by:
$$
A=\sum_{n\ge 0} A_n^{FC}(q) x^n, \qquad B=\sum_{n\ge 0} B_n^{FC}(q)
x^n, \qquad D=\sum_{n\ge 0} D_{n+1 }^{FC}(q) x^n.
$$
Then:
\beq\label{A-expr}
A= \frac 1 {1-xq}\frac{\JJ(xq)}{\JJ(x)},
\eeq
$$
B= \frac {\KK(x)}{\JJ(x)} + \frac{xq^2(1-x)}{(1-xq)(1-xq^2)} \frac{\JJ(xq)}{\JJ(x)}- \frac{xq^2}{1-xq^2},
$$
and
$$
D= 
\frac {2\KK(x)}{\JJ(x)} + \frac{q-(1+q)x+x^2q^2}{(1-xq)(1-xq^2)}
\frac{\JJ(xq)}{\JJ(x)}- \frac{q}{1-xq^2}-1, 
$$
where the series $\JJ$ and $\KK$ are defined by~\eqref{JK-def}.
\end{Theorem}
\noindent{\bf Remarks}\\
1. Above, we have adopted the convention that
$B_0=A_0=\{\id\}$ and $B_1\,  {\cong}\, 
A_1=\{\id, s_1\}$, so that
$B_0^{FC}(q)= A_0^{FC}(q)=1 $ and $B_1^{FC}(q)= A_1^{FC}(q)=1+q $. For
$n=0$ or 1, the group $D_{n+1}$ is not well defined, and the coefficients
of $x^0$ and $x^1$ in $D(x,q)$ are irrelevant. {We could
  thus drop the term $-1$ in the expression of $D$. The only reason
  why it is there is to fit with the series $D^{FC}$
  of~\cite[Prop.~4.6]{BJN-long}. With this convention, $D_1^{FC}(q)=1$
  and $D_2^{FC}(q)=(1+q)^2$.}

\smallskip \noindent
{2. One can feed a computer algebra system with the above
  expressions and expand them in $x$ to compute the polynomials
  $W_n^{FC}(q)$ for small values of $n$. This gives for instance the
  values of Table~\ref{tab:finite}.

This also allows us to check our results, since we can compute
independently
 the series $W_n^{FC}(q)$ (or at least its first
coefficients) for  any given
  Coxeter group $W_n$:  according to~\cite[Sec.~6]{BJN-long}, fc elements of $W_n$ index a basis 
  of the so-called \emm nil-Temperley--Lieb algebra, associated with
  $W_n$. This algebra is graded, and in fact $W_n^{FC}(q)$ is its
  Hilbert series.  For small values of $n$ and $\ell$,  the package {\sc{GBNP}} of
  {\sc{GAP}} can compute a basis of the grade-$\ell$ component of the
  algebra: this is nothing but the list of fc elements of length $\ell$
  in $W_n$, and the number of such elements is the coefficient of
  $q^\ell$ in $W_n^{FC}(q)$. We can then select the involutions among
  these elements to
  determine the first terms of $\cW_n^{FC}(q)$.}

\begin{table}[htb]
  \centering
  \begin{tabular}{|c|ccc|}
 $n$ & 2 &3 & 4 \\
\hline &&&\\
$A_{n}^{FC}$    &[1, 2, 2] &[1, 3, 5, 4, 1]&[1, 4, 9, 12, 10, 4, 2]\\
$B_{n}^{FC}$ & [1, 2, 2, 2]&[1, 3, 5, 6, 5, 3, 1]&[1, 4, 9, 14, 16, 15, 11, 7, 3, 2, 1]\\
$D_{n+1}^{FC}$&[1, 3, 5, 4, 1]&[1, 4, 9, 13, 11, 7, 3]&[1, 5, 14, 26, 34, 32, 25, 17, 7, 4, 2]
  \end{tabular}
\vskip 3mm
  \caption{Length generating functions of fc elements in finite Coxeter
    groups. The list $[a_0, \ldots, a_k]$ stands for the polynomial
    $a_0+a_1q+\cdots + a_k q^k$.}
  \label{tab:finite}
\end{table}

\smallskip
\noindent {3. In the series $\JJ(x)$ and $\KK(x)$, the coefficient of $x^n$ is a
\emm rational function, of $q$ whose poles are roots of unity, whereas
in $A, B$ and $D$, the coefficient of $x^n$ is a \emm
polynomial, in $q$ (because there are finitely many fc elements in each
group $A_n, B_n$ and $D_n$). In Section~\ref{sec:column} we derive alternative
expressions of $A, B$ and $C$, in which the coefficient of $x^n$ appears
as a \emm Laurent polynomial, in $q$. See for instance
{Theorem}~\ref{thm:A-neg} for the $A$-case.}

\medskip
The next proposition clarifies the algebraic properties of the series
$\JJ$ and $\KK$. Both satisfy linear $q$-equations, of respective
orders 2 and 3.  By analogy
with the theory of D-finite series~\cite[Chap.~6]{stanley-vol2}, we
could say that $J$ and $K$  are \emm $q$-finite,. We refer to~\cite{chyzak-salvy}
for an algebraic treatment of such series in terms of Ore
algebras, {where the \emm $q$-shift, $F(x) \mapsto F(xq)$
  plays the role of a derivation. Thanks to the following
proposition,  one can conclude 
that, if $G(x)$ is any of the above \gf\ of fc elements, $G(x)\JJ(x)$ belongs to
a 3-dimensional vector space over $\qs(x,q)$ closed under the $q$-shift.}

To study the series $J$ and $K$,
it is convenient to introduce another series $\HH$,
closely related to~$\JJ$:
\beq\label{H-def}
\HH(x)= \sum_{n \ge 0} \frac{(-x)^n q^{n\choose 2}}{(q)_n(x)_n}.
\eeq

\begin{Proposition}\label{prop:JHK}
 The following identities hold:
\beq\label{eq-JH}
\JJ(x)=  ( 1-x ) H  ( x ) +xH  ( xq ),
\eeq
and
\beq\label{eq-HJ-alt}
\JJ(x)+\frac x{1-xq}\JJ(xq)=\HH(xq).
\eeq
Consequently, 
\beq\label{eq-HJ}
H  ( x ) =\JJ(x)- {\frac {{x}^{2}J  ( xq ) }{  ( 1-x ) (1- xq
 )   }},
\eeq
and the series $\HH(x)$ and $\JJ(x)$  satisfy
 second order linear $q$-equations:
\beq\label{eqH-lin}
( 1-x )  H ( x ) - ( 1-2x ) H ( xq ) +\frac{{x}^{2}q}{1-xq}H( x{q}^{2} ) =0,
\eeq
\beq\label{eqJ-lin}
  (1-xq)\JJ(x)-(1-x(1+q))\JJ(xq)+\frac{x^2q^2}{1-xq^2}\JJ(xq^2)=0.
\eeq
 The vector space (over $\qs(x,q)$) spanned by $\JJ(x)$
and its $q$-shifts $\JJ(xq), \JJ(xq^2), \ldots$ has dimension~$2$, and
coincides with the space spanned by all series $H(xq^i)$.

We also have
\beq\label{eq-KH}
K(x)-\frac{xq}{1-xq} K(xq)= H(xq),
\eeq
so that
 $K(x)$ satisfies a linear $q$-equation of order $3$:
$$
\KK(x)-\KK(xq)+\frac{xq^2}{1-xq^2}\KK(xq^2)-\frac{x^3q^6}{(1-xq)(1-xq^2)(1-xq^3)}
\KK(xq^3)=0.
$$ 
The vector space
spanned by the series $\KK(xq^i)$ over $\qs(x,q)$ has dimension $3$, and
contains all $q$-shifts of $\JJ$ and $H$ by~\eqref{eq-KH}.
\end{Proposition}
\begin{proof}
 We begin with the first two identities:
\begin{align*}
  \JJ(x)-(1-x)\HH(x) &=  \sum_{n\ge 0}\frac{(-x)^nq^{n\choose
      2}}{(q)_n(xq)_n} - \sum_{n\ge 0}{ \frac{(-x)^{n}q^{n\choose
      2}(1-x)}{(q)_n(x)_{n}}}\\
&=  \sum_{n\ge 0}\frac{(-x)^nq^{n\choose
      2}}{(q)_n(xq)_n} \left(1-(1-xq^n)\right) = x \HH(xq),
\end{align*}
and
\begin{align*}
  \JJ(x)+\frac x{1-xq}\JJ(xq) &= \sum_{n\ge 0}\frac{(-x)^nq^{n\choose
      2}}{(q)_n(xq)_n} - \sum_{n\ge 0} \frac{(-x)^{n+1}q^{n+1\choose
      2}}{(q)_n(xq)_{n+1}}\\
&= \sum_{n\ge 0}\frac{(-x)^nq^{n\choose
      2}}{(q)_n(xq)_n} - \sum_{n\ge 1} \frac{(-x)^{n}q^{n\choose
      2}}{(q)_{n-1}(xq)_{n}}\\
&= \sum_{n\ge 0}\frac{(-x)^nq^{n\choose
      2}}{(q)_n(xq)_n} \left(1 -(1-q^n)\right) = \HH(xq).
\end{align*}
By eliminating $H(xq)$ between~\eqref{eq-JH} and~\eqref{eq-HJ-alt}
 we obtain~\eqref{eq-HJ}. Also, by eliminating $J(x)$ and $J(xq)$
 between~\eqref{eq-JH},~\eqref{eq-HJ-alt}, and the shifted version
 of~\eqref{eq-JH}, we obtain  the second order equation~\eqref{eqH-lin} satisfied by
 $H$. 
 By eliminating $H(xq)$ between~\eqref{eq-HJ-alt} and
 the shifted version of~\eqref{eq-HJ}, we obtain the linear equation~\eqref{eqJ-lin}
 satisfied by $J$.

 Let us now prove~\eqref{eq-KH}: 
\begin{align*}
 K(x)-\frac{xq}{1-xq} K(xq)&= \sum_{n\ge 0} \frac{x^n q^{n+1\choose
    2}}{(xq)_n} \sum_{k=0}^n \frac{(-1)^k}{(q)_k} 
- \sum_{n\ge 1} \frac{x^n q^{n+1\choose
    2}}{(xq)_n} \sum_{k=0}^{n-1} \frac{(-1)^k}{(q)_k} \\
&= \sum_{n\ge 0} \frac{(-x)^n q^{n+1\choose
    2}}{(xq)_n(q)_n}= H(xq).
\end{align*}
{Finally,} combining~\eqref{eq-KH} with {the shifted version of}~\eqref{eqH-lin} gives the third order equation
satisfied by $K$.

It remains to prove our results about dimensions.  If the space
spanned by the $q$-shifts of $J(x)$ had dimension 1 only,
$J$ would satisfy a first order linear equation, so
that $J(x)/J(xq)$ would be a rational function. Assume this is
the case, and write
$$
(1-xq) \frac{J(x)}{J(xq)}= \frac{N(x)}{D(x)},
$$
where $N$ and $D$ are polynomials in $x$ with coefficients in
$\qs(q)$,  with no common
factor. Dividing~\eqref{eqJ-lin} by $J(xq)$ then gives
$$
\frac{N(x)}{D(x)}+ x^2q^2 \frac{D(xq)}{N(xq)}= 1-x(1+q).
$$
{Given that $J(x)=1+O(x)$, we {can} normalize $N$ and $D$ by fixing
$N(0)=D(0)=1$.  {As $N$ and $D$ are relatively prime}, the above identity then implies that}
$D(x)=N(xq)$. Thus the 
polynomial $N(x)$ must satisfy
$$
N(x)+x^2q^2N(xq^2)= (1-x(1+q)) N(xq),
$$
but considering the degree in $x$ shows that this equation has no polynomial
solution, except $N(x)=0$. {Hence $J(x)$ and its $q$-shifts span a
2-dimensional space, and the same holds for $H(x)$ thanks to~\eqref{eq-HJ}
and~\eqref{eq-JH}.} 

\medskip
Finally, let us prove that 
{the space spanned by the shifts of $K$ cannot have dimension less than
3. In that case, it would have dimension 2 
(since it contains $H(x)$ and its $q$-shifts, by~\eqref{eq-KH}) and there
would exist rational functions $\alpha$ and $\beta$ such that}
$$
K(x)= \alpha(x) H(x) + \beta(x) H(xq).
$$
By combining~\eqref{eq-KH} and~\eqref{eqH-lin}, this would imply
$$
H(x)\left( \al(x)+{\frac{\beta(xq)(1-x)}x} \right)+ H(xq)\left( \be(x)
  -\frac{xq\, \al(xq)}{1-xq} -\frac{\be(xq)(1-2x)}x-1\right)=0,
$$
from which we derive 
{$\alpha(x)=-(1-x)\beta(xq)/x$ and}
$$
\be(x) {+\be(xq^2)}
-\frac{\be(xq)(1-2x)}x-1=0.
$$
It remains to apply Abramov's algorithm~\cite{abramov}, which determines all rational
solutions of a linear $q$-equation, to conclude that such a $\beta$ does
not exist  (we have used the {\sc Maple} implementation of Abramov's
algorithm, via the 
{\tt RationalSolution} command of the {\tt QDifferenceEquations}
package).
\end{proof}

\subsection{Fully commutative involutions}\label{sec:fci}
We now state analogous results for fc involutions. With the
notation~\eqref{double}, the main two series are now: 
\beq\label{DD-def}
\JI(x)=\sum_{n\geq0}\frac{(-1)^{\lceil n/2\rceil} x^nq^{n \choose
    2}}
{\llp q^2 \rrp _{\lfloor n/2\rfloor}},
\eeq
and
\beq\label{UV-def}
\KI(x)= \sum_{n\geq0}x^{n}q^{n+1\choose 2}\sum_{k=0}^{\lfloor n/2
\rfloor}\frac{(-1)^k} {\llp q^2\rrp_k}.
\eeq
For any  series $F(x)$, we denote by $F_e(x)$ and $F_o(x)$ its even
and odd parts in $x$:
\beq\label{eop}
F_e(x)= \frac 1 2 \left( F(x)+F(-x)\right), \qquad 
F_o(x)= \frac 1 2 \left( F(x)-F(-x)\right).
\eeq

\begin{Theorem}\label{thm:ABD-inv}
  Let $\cA(x,q)\equiv \cA$, $\cB(x,q)\equiv \cB$, and $\cD(x,q)\equiv \cD $ be the
  \gfs\  of fully commutative involutions of types $A$, $B$ and $D$, defined
  respectively by:
$$
\cA=\sum_{n\ge 0} \cA_n^{FC}(q) x^n, \qquad \cB=\sum_{n\ge 0} \cB_n^{FC}(q)
x^n, \qquad \cD=\sum_{n\ge 0} \cD_{n+1 }^{FC}(q) x^n.
$$ 
Then:
\beq\label{cA-expr}
\cA= \frac{\JI(-xq)}{\JI(x)},
\eeq
$$
\cB
=
\frac{\KI(x)}{\JI(x)}+\frac{xq^2(1-x)}{1-xq^2}
\frac{\JI(-xq)}{\JI(x)} - \frac{xq^2}{1-xq^2},
$$
and
$$
\cD=
\frac{2xq \KI_e(xq)}{\JI(x)}+  
 \frac{q+x(1-q)-x^2q^2}{1-xq^2}\frac{\JI(-xq)}
{\JI(x)}- { \frac {q} {1-xq^2} +1},
$$
where the series $\JI$ and $\KI$ are defined by~\eqref{DD-def} and~\eqref{UV-def}.
\end{Theorem}
\noindent{\bf Remark.}
 Again, the coefficients of $x^0$ and $x^1$ in
$\cD(x,q)$ are irrelevant, {but in agreement with the convention
of~\cite{BJN-inv}. They read $\cD_1^{FC}(q)=1$
  and $\cD_2^{FC}(q)=(1+q)^2$. Table~\ref{tab:finite-inv} lists the
  values of $\cW_n^{FC}(q)$ for small values of $n$.}

\begin{table}[htb]
  \centering
  \begin{tabular}{|c|ccc|}
 $n$ & 2 &3 & 4 \\
\hline &&&\\
$\cA_{n}^{FC}$&  [1, 2]& [1, 3, 1, 0, 1]& [1, 4, 3, 0, 2]
\\
$\cB_{n}^{FC}$ &[1, 2, 0, 2]& [1, 3, 1, 2, 1, 1, 1]& [1, 4, 3, 2, 4, 1, 3, 1,
                                              1, 0, 1]
\\
$\cD_{n+1}^{FC}$&[1, 3, 1, 0, 1]& [1, 4, 3, 1, 3, 1, 3]& [1, 5, 6, 2, 4, 2,
                                                   3, 1, 1]
  \end{tabular}
\vskip 3mm
  \caption{Length generating functions of fc involutions in finite Coxeter
    groups. The list $[a_0, \ldots, a_k]$ stands for the polynomial
    $a_0+a_1q+\cdots + a_k q^k$.}
  \label{tab:finite-inv}
\end{table}

\noindent 
\begin{Proposition}
   The series $\JI$  satisfies a linear $q$-equation of order $2$:
\beq\label{eq-UU}
(1-2xq)\JI(x)-(1-x(1+q)) \JI(xq)+x^2q^2(1-2x)\JI(xq^2)=0.
\eeq
We also have
\begin{align}
  (1-2x)\JI(-xq)&=2\JI(x)-\JI(xq), \label{JI-minus}
\\
\label{JI-JIe}
\JI(x)&=\JI_e(x)-x\JI_e(xq),
\end{align}
and 
\beq\label{HI-eq}
\HI(x)-\HI(xq)+x^2q\HI(xq^2)=0.
\eeq
The vector space (over $\qs(x,q)$) spanned by $\JI(x)$
and its $q$-shifts $\JI(xq), \JI(xq^2), \ldots$ has dimension~$2$, and
coincides with the space spanned by all series $\JI(-xq^i)$, and with
the space spanned by all series $\JI_e(xq^i)$.

The series $\KI$ is related to $\JI$ by: 
\beq\label{KI-JI}
\KI(x){-} xq\KI(xq)=\frac{\JI(x)-x\JI(xq)}{1-2x}= \JI(x)+x\JI(-xq),
\eeq
and thus satisfies a $q$-equation of order $3$:
\beq\label{KIlin}
\KI(x)-(1+xq)\KI(xq)+xq^2(1+xq)\KI(xq^2)-x^3q^6\KI(xq^3)=0.
\eeq
  The series $\KI(xq^i)$ span a $3$-dimensional vector space, which 
contains   $\JI(x)$ since
\beq\label{JI-KI}
\JI(x)=(1-x)\KI(x)-qx\KI(xq)+x^3q^3\KI(xq^2).
\eeq
Finally, 
\beq\label{KI-KIe}
\KI(x)= \KI_e(x)+xq\KI_e(xq),
\eeq
and  the series $\KI_e(xq^i)$ span a $4$-dimensional vector space, which 
contains $\KI(x)$ and $\JI(x)$.
 \end{Proposition}
\begin{proof}
 Since the expansions in $x$ of the series 
$\JI$ and $\KI$ are given explicitly, all identities above  are
 readily proved by extracting the   coefficient of $x^n$, for $n \in \ns$.

It remains to prove our results about dimensions. The arguments are
the same as in the last part of the proof of Proposition~\ref{prop:JHK}. If
the space spanned by the $\JI(xq^i)$ had dimension~1, then 
$\JI(x)/\JI(xq)$ would be a rational function $N(x)/D(x)$, with
$N(x)$ and $D(x)$ coprime in $\qs(q)[x]$ and $N(0)=D(0)=1$. Then dividing~\eqref{eq-UU} by $\JI(xq)$ would give 
\beq\label{eqDN}
(1-2xq) \frac{N(x)}{D(x)} +x^2q^2(1-2x) \frac{D(xq)}{N(xq)}=1-x(1+q).
\eeq
This gives four possible relations between {the polynomials} $N$ and $D$:
$$
D(x)=N(xq), \qquad D(x)=(1-2xq) N(xq), \qquad D(x)
=\frac{N(xq)}{1-2x}, \qquad D(x)
=\frac{(1-2xq)N(xq)}{1-2x}.
$$
Reporting each of these possibilities in~\eqref{eqDN} gives a linear
$q$-equation for $N$, and a degree argument shows that this equation has no
non-zero solution {(in the third case, the degree argument
  does not suffice, and we have to extract the dominant coefficient in
  $x$ to
  conclude). Hence $\JI(x)$ and its $q$-shifts span a 2-dimensional
space.  The same space is spanned by the series $\JI(-xq^i)$, thanks
to~\eqref{JI-minus}, and finally by the series  $\JI_e(xq^i)$, thanks
to~\eqref{JI-JIe} and~\eqref{HI-eq}}.

{Now assume that the space spanned by the series $\KI(xq^i)$ has  dimension less
than~3. Then its dimension is 2,  since it contains
$\JI(x)$ by~\eqref{JI-KI}, and  there exist rational functions $\al$ and
$\be$ such that}
$$
\KI(x)= \alpha(x) \JI(x) +\beta(x) \JI(xq).
$$
 Then~\eqref{KI-JI} and~\eqref{eq-UU} give
\begin{multline*}
  \Big( xq (1- 2x ) \al  ( x ) + (1 -2xq
 ) \be  ( xq ) -xq \Big)\JI ( x ) \\+
 \Big( {x}^{2}{q}^{2} ( 2x-1 ) \al  ( xq ) -xq
 ( 2x-1 ) \be  ( x ) + ( xq+x-1 ) \be 
 ( xq ) +{x}^{2}q \Big) \JI ( xq ) =0
,
\end{multline*}
from which it follows that
$$
xq (1- 2x )  ( 2xq-1 ) \be  ( x ) +
 ( xq+x-1 )  ( 2xq-1 ) \be  ( xq ) -x
 ( 2x-1 )  ( 2x{q}^{2}-1 ) \be  ( x{q}^{2}
 ) ={x}^{2}q (1- q ) 
.
$$
But applying Abramov's algorithm proves that this equation has no
rational solution.

Finally, {assume that the space spanned by the series $\KI_e(xq^i)$ has
dimension less than  4. Then  it has dimension~3  (since it contains
$\KI(x)$ by~\eqref{KI-KIe}) and there  exist rational functions $\al$,
$\be$ and $\gamma$ such that}
$$
\KI_e(x)= \alpha(x) \KI(x) +\beta(x) \KI(xq)+\gamma(x) \KI(xq^2).
$$
 Then
we derive from~\eqref{KI-KIe}
and~\eqref{KIlin} a linear $q$-equation for $\gamma(x)$:
$$
{q}^{6}x\gamma ( x ) +{q}^{3} ( xq+1 ) \gamma
 ( xq ) + ( x{q}^{2}+1 ) \gamma ( x{q}^{2}
 ) +
 {x}\gamma ( x{q}^{3} ) ={x}^{3}{q}^{9}
.
$$
Applying again Abramov's algorithm  proves that this equation has no
rational solution.
\end{proof}

\section{First  recursive approach: Peeling a diagonal}
\label{sec:recursive}
In this section, we prove the finite type results stated in the
previous section. We use the 
 description of fc elements in terms of
heaps given in~\cite{BJN-long,St1}. 
Then the central question is to count alternating heaps over a path, having at most one point in the rightmost column. 
Our approach is recursive, and consists in peeling the rightmost NW-SE
diagonal of alternating heaps. A precise description is given below,
but we refer to Figure~\ref{fig:A} for a quick intuition.  This is an
adaptation of a 
classical method used to count \emph{convex
polyominoes}~\cite{klarner-rivest-conv,bousquet-vcd}.

\subsection{Type  $\boldsymbol A$}
\label{sec:peel-A}
 
It is known from~\cite{St1} that fc elements in $A_n$ are in
bijection with alternating heaps over the $n$-point
path, having at most one point {in their first (i.e. leftmost) and last (rightmost)
columns.
}
We count them by recording three parameters:
 the number $n$
of generators of the group (variable $x$), the length of the fc element
(variable $q$), and   the size $i$ of its largest right factor
$s_{n-i+1}\cdots s_n$ (variable~$s$). In graphical terms, $i$ is the
size of the rightmost NW-SE diagonal of the heap (Figure~\ref{fig:A}).  We denote by $A(s)\equiv A(s;x,q)$ the corresponding \gf, and by
  $A_i\equiv A_i(x,q)$ the coefficient of $s^i$ in this series:
$$
A(s)\equiv A(s;x,q)=\sum_{i\ge 0} A_i s^i .
$$
Thus the series $A(x,q)$ of {Theorem}~\ref{thm:A-tA} is now
$A(1;x,q)\equiv A(1)$. We hope that this will not cause any
confusion.

We count separately four types of heaps:
  \begin{itemize}
  \item the trivial group ($n=0$) contributes $1$; we assume in
    what follows that $n\ge 1$; 
\item if the last column is empty, that is, $i=0$, the heap encodes an
  fc element  of $A_{n-1}$. Thus the \gf\ for this type  is simply
  $xA(1)$;
\item if the last column contains a point, and this point is lower than any
  point in the next-to-last column, removing this rightmost
  point leaves an fc element of $A_{n-1}$ (Figure~\ref{fig:A},
  left; {this includes the case where the next-to-last
    column is empty}). The size of the rightmost diagonal decreases by $1$. Hence
  the \gf\ for this type is $xsqA(s)$;
\item finally, if the last column contains a point, and this point is higher
  than the lowest point of the next-to-last column, we remove the top
  point in each of the $i$ rightmost columns (Figure~\ref{fig:A},
  right). {In other words, we peel off the rightmost NW-SE
    diagonal.} This leaves an fc element of
  $A_{n-1}$, with a rightmost diagonal of size $j \ge i$.  This shows   that the \gf\ for   this type is
$$
x\sum_{j\ge 1} A_j \sum_{i=1}^j (sq)^i= x\sum_{j\ge 1} A_j
\frac{sq-(sq)^{j+1}}{1-sq}= \frac {xsq}{1-sq}\left(A(1)-A(sq)\right).
$$
\end{itemize}
\begin{figure}[ht]
\begin{center}
{\scalebox{1}{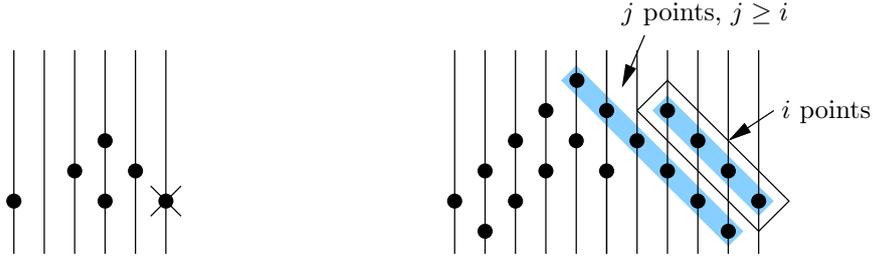}}
\caption{Two types of alternating heaps with one point in the last column.} 
\label{fig:A}
\end{center}
\end{figure}
Putting together the four cases  gives:
\beq\label{A-eq}
A(s)= 1+xA(1)+xsqA(s)+ \frac {xsq}{1-sq}\left(A(1)-A(sq)\right).
\eeq
Grouping the terms in $A(s)$ and in $A(1)$ and dividing by $(1-xsq)$ yields:
$$
A(s)=  \frac 1 {1-xsq}+ \frac x {(1-sq)(1-xsq)}A(1)- \frac {xsq}
{(1-sq)(1-xsq)}A(sq).
$$
Iterating this equation $m$ times provides an expression of $A(s)$ in
terms of $A(sq^{m+1})$:
\begin{multline*}
  A(s)= \sum_{n= 0}^m \frac{(-xs)^n q^{n+1\choose 2}}{(sq)_n(xsq)_n}
\left( \frac 1 {1-xsq^{n+1}}+ \frac x
     {(1-sq^{n+1})(1-xsq^{n+1})}A(1)\right)\\
+\frac{(-xs)^{m+1} q^{m+2\choose 2}}{(sq)_{m+1 }(xsq)_{m+1}} A(sq^{m+1}).
\end{multline*}
As a series in $x$, the rightmost term tends to $0$ as $m$ tends to
infinity. Hence:
$$
A(s)= \sum_{n\ge 0} \frac{(-xs)^n q^{n+1\choose 2}}{(sq)_n(xsq)_n}
\left( \frac 1 {1-xsq^{n+1}}+ \frac x {(1-sq^{n+1})(1-xsq^{n+1})}A(1)\right).
$$
Setting $s=1$ gives
\begin{align*}
  A(1)&= \sum_{n\ge 0} \frac{(-x)^n q^{n+1\choose 2}}{(q)_n(xq)_{n+1}}
-
A(1) \sum_{n\ge 0} \frac{(-x)^{n+1} q^{n+1\choose
    2}}{(q)_{n+1}(xq)_{n+1}}\\
&=\frac{J(xq)}{1-xq} - A(1) \left( J(x)-1\right),
\end{align*}
with $J(x)$ defined by~\eqref{JK-def}.
  Solving for $A(1)$ gives the first result of
{Theorem}~\ref{thm:ABD}. \qed

\bigskip
We now adapt this to fc involutions of  $A_n$.  They are in bijection with
self-dual alternating heaps over the $n$-point path in which the first
and last columns contain at most one point. In order to preserve
self-duality,  the peeling procedure must now
remove the rightmost NW-SE diagonal \emm and the rightmost SW-NE
diagonal, (Figure~\ref{fig:A-inv}). We denote by
$\cA(s;x,q)=\sum_{i\geq0} \cA_i s^i$ the  \gf\ of fc involutions, refined by the same
parameter as above. We count separately  the same four types of heaps,
but there are no self-dual heaps of the third type:
\begin{itemize}
\item the empty group ($n=0$) contributes $1$; we assume from now on that $n\geq1$;
\item if the last column is empty,  that is, $i=0$, the generating function is $x\cA (1)$;
\item if the last column contains one point, that is, $i\geq1$, there
  are two cases: either  $n=1$, and then the contribution is $xsq$, or
  $n\geq2$. In this case, removing the rightmost  
{NW-SE and SW-NE diagonals} leaves an fc involution of
  $A_{n-2}$, with a rightmost diagonal of size $j \ge i-1$
  (Figure~\ref{fig:A-inv}). Hence  the generating function for this type reads:
$$
{xsq+} x^2\sum_{j\geq0 }\cA_j \sum_{i=1}^{j+1}s^iq^{2i-1}
={xsq+} \frac{x^2sq}{1-sq^2}\left(\cA (1)-sq^2\cA (sq^2)\right).
$$
\end{itemize}
Putting all contributions together and gathering the terms $\cA(1)$ gives:
\beq\label{eq-A-inv}
\cA (s)
=1+xsq+x\left(1+\frac{xsq}{1-sq^2}\right)\cA
(1)-\frac{x^2s^2q^3}{1-sq^2}\, \cA (sq^2).
\eeq
Iterating $m$ times this equation and then letting $m$ tends to infinity yields
$$
\cA(s)=\sum_{n\geq0}\frac{(-x^2s^2)^nq^{n(2n+1)}}{
\llp sq^2 \rrp_n
}
\left(1+xsq^{2n+1}+x\left(1+\frac{xsq^{2n+1}}{1-sq^{2n+2}}\right)\cA
  (1)\right).
$$
{Once $s$ is set to 1, this reads, with the notation~\eqref{DD-def}:
$$
\cA(1)= \JI(-xq) - (\JI(x)-1) \cA(1).
$$}
Solving for $\cA(1)$ gives the first result of
{Theorem}~\ref{thm:ABD-inv}. \qed

\begin{figure}[h!]
\begin{center}
{\scalebox{1}{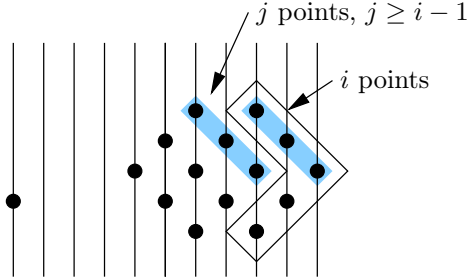}}
\caption{Peeling the rightmost diagonals in a self-dual alternating heap.}
\label{fig:A-inv} 
\end{center}
\end{figure}

\bigskip
\noindent{\bf Remark: connection with staircase polyominoes and Dyck paths.} We have proved these  $A$-results to illustrate
our recursive approach, but they are equivalent to previously
published results involving staircase polyominoes and Dyck paths,
respectively. More precisely, let us say that a non-empty fc element
of $A_n$ is \emm connected, (or: \emm has full support,) if
every generator occurs in it. In this case, if we replace every point of the
corresponding heap by a unit square, the resulting collection of
squares is a \emm staircase  
polyomino, (Figure~\ref{fig:fs}, left). The perimeter of this polyomino is $2n+2$,
and its area is 
the length of the fc element. Let $P(x,q)$ be the \gf\ of
staircase polyominoes, counted by the half-perimeter (variable $x$)
and the area (variable $q$). Then  the \gf\
of connected fc elements of type $A$ is  $\Ac(x,q)=  P(x,q)/x$. By
discussing whether the first 
column of an fc heap of type $A$ is empty or not, one can relate the
series $A$ and $\Ac$ as follows:
$$
A= 1+ xA +\Ac + \Ac x A,
$$
so that
$$ 
A= \frac{1+\Ac}{1-x-x\Ac} = \frac{1+P/x}{1-x-P}.
$$ 
An expression of $P(x,q)$ can be found
in~\cite[Thm.~3.2]{bousquet-vcd}
or~\cite[Prop.~4.1]{bousquet-viennot}. With our notation, it reads
$$
P(x,q)= \frac{x^2q}{1-xq}\frac{H(xq^2)}{H(xq)},
$$
where $H(x)$ is defined by~\eqref{H-def}.  One then  recovers the expression of $A$ given in
{Theorem}~\ref{thm:ABD} using~\eqref{eq-JH} and~\eqref{eq-HJ-alt}.

\begin{figure}[ht]
\begin{center}
\includegraphics[scale=0.8]{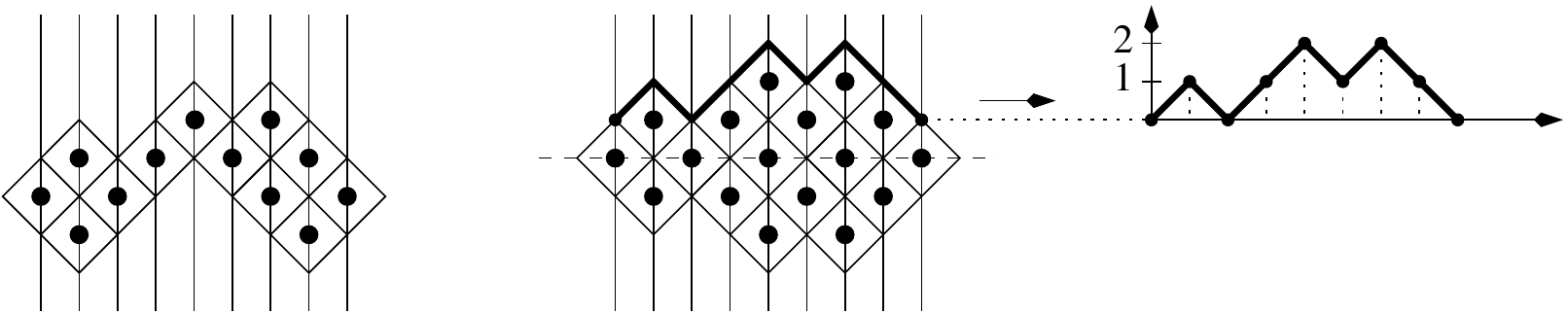}
\caption{Left: A connected fc element of  $A_9$ and the
  corresponding staircase polyomino, formed of square cells. Right: A connected
fc involution of $A_9$ and the corresponding Dyck path, with the
heights of the vertices shown.}
\label{fig:fs}
\end{center}
\end{figure}

Fully commutative involutions of type $A$ are related in  a similar
fashion to  staircase polyominoes  that are invariant by reflection
in a horizontal line, with half-perimeter/area \gf\ $\cP(x,q)$:
$$
\cA=  \frac{1+\cP/x}{1-x-\cP}.
$$ 
These  polyominoes are in bijection with the famous \emm Dyck
paths, (Figure~\ref{fig:fs}, right). More precisely, a symmetric
staircase polyomino of half-perimeter $n$ {(necessarily even)} and area $a$ gives rise to a
Dyck path of length $n-2$ in which the \emm total height, (the sum of heights of the vertices)
is $a-n+1$. Hence, denoting by 
$D(x,q)$ the \gf\ of Dyck paths, counted by half-length and total height, we
have
$$
 \cP{(x,q)}= x^2 q\,D(x^2q^2,q).
$$
An expression for $D(x,q)$ can be found
in~\cite[Ex.~5.2.12(c)]{goulden-jackson}. With our notation, it gives
$$
D(x^2q^2,q)= \frac{\HI(xq^2)}{\HI(xq)},
$$
where $\HI(x)$ is the even part of the series $\JI(x)$ defined
by~\eqref{DD-def}.  One then  recovers the expression of $\cA$ given in
{Theorem}~\ref{thm:ABD-inv} using~\eqref{JI-JIe} and~\eqref{HI-eq}.

\subsection{Type  $\boldsymbol B$}
\label{sec:recB}
As argued in~\cite[Sec.~4.4]{BJN-long},  fc elements of
 $B_n$ come in two types:
\begin{itemize}
\item first, we have all alternating heaps over the $n$-point path having at most
  one point in the last column. We denote by $\Ba(x)$ their \gf;
\item the remaining fc elements are not alternating. With the notation
  of Figure~\ref{fig:dynkin}, they are obtained 
as follows: {one starts from an
alternating heap over a path with vertices $s_j, \ldots,
  s_{n-1}$ (with $1\le j\le n-1$), having  exactly one point in
  the $s_j$-column and at most one point in the last column, and
  inflates the point in the $s_j$-column into a 
$<$-shaped heap
  $s_j s_{j-1} \cdots s_1 t s_1 \cdots s_{j-1} s_j$
  (see the third picture of~\cite[Fig.~7]{BJN-long} for an illustration)}. We denote by $\Bna(x)$ the
  corresponding \gf. 
\end{itemize}
The latter description gives $\Bna(x)=xq^2/(1-xq^2) A^{(1)}(x)$, where $ A^{(1)}(x)$ counts
alternating heaps {over a path} with one point in the first column and at most
one point in the last one. Since the \gf\ of alternating heaps
with an empty first column is $xA$, we have
$A^{(1)}(x)= A-1-xA$. Thus:
\begin{align}
\Bna(x)&= \frac{xq^2}{1-xq^2} \left((1-x)A-1\right) \label{Bna-A}
\\
&=
 \frac{xq^2}{1-xq^2}\left( \frac{1-x}{1-xq} \frac{\JJ(xq)}{\JJ(x)}
   -1\right).  \label{Bna-sol}
\end{align}

We now focus on {the \gf\  $\Ba(x)$ of} alternating fc
elements of type $B$.  We {refine} 
 their enumeration by recording the size $i$ of  
the rightmost diagonal: with the  generators denoted as in Figure~\ref{fig:dynkin}, this is
the longest right factor of the heap of the form $s_{n-i}\cdots s_{n-1}$, with $s_0=t$. 
We  apply  the recursive
approach that led to~\eqref{A-eq} in Section~\ref{sec:peel-A}. {The
only difference is that we can now have several points in the first
column.} The first three cases  contribute as
before (with $A(s)$ replaced by $\Ba(s)$). However, the fourth case is
now richer: {when the point in the last column is higher
  than some point in the next-to-last column},
removing the top point in each of the $i$ rightmost columns may leave
an alternating heap with a rightmost diagonal of size $i-1$ {(instead
of size $j\ge i$ in the $A$-case). This
happens only if $i$ is maximal, that is, equal to the number of
generators. An example is shown in Figure~\ref{fig:T}, left. We shall see below that the contribution of these heaps
is $xsqT(sq)$, where
\beq\label{eq-T}
T(s)=xsq+xsqT(s)+xsqT(sq),
\eeq
which implies
\beq\label{Tsol}
T(s)= \sum_{m\ge 1} \frac{(xs)^m q^{m+1\choose 2}}{(xsq)_m}.
\eeq}
Putting together all four cases gives:
\beq\label{Ba-eq}
\Ba(s)= 1+x\Ba(1)+xsq\Ba(s)+ \frac {xsq}{1-sq}\left(\Ba(1)-\Ba(sq)\right)+ xsqT(sq).
\eeq
We now proceed as we did for the \bjnChange{A}-equation~\eqref{A-eq}. Grouping the
terms in $\Ba(s)$ and in $\Ba(1)$, and dividing by $(1-xsq)$, gives:
\begin{align*}
  \Ba(s)&=  \frac 1 {1-xsq}+ \frac{xsq}{1-xsq}T(sq)+ \frac x
{(1-sq)(1-xsq)}\Ba(1)
- \frac {xsq}{(1-sq)(1-xsq)}\Ba(sq)
\\
&=  1+T(s)+ \frac x
{(1-sq)(1-xsq)}\Ba(1)
- \frac {xsq}{(1-sq)(1-xsq)}\Ba(sq)
\end{align*}
by~\eqref{eq-T}.
Iterating the equation gives
\beq\label{B-itere}
\Ba(s)= \sum_{n\ge 0} \frac{(-xs)^n q^{n+1\choose 2}}{(sq)_n(xsq)_n}
\left( 1+T(sq^{n})+ 
\frac x {(1-sq^{n+1})(1-xsq^{n+1})}\Ba(1)\right).
\eeq
Using the expression~\eqref{Tsol} of $T(s)$, we can rewrite
\begin{align*}
  \sum_{n\ge 0} \frac{(-xs)^n q^{n+1\choose 2}}{(sq)_{n}(xsq)_n}
\left(1+T(sq^{n})\right)&=
 \sum_{n\ge 0} \frac{(-xs)^n q^{n+1\choose 2}}{(sq)_{n}(xsq)_n}
\sum_{m\ge 0} \frac{(xsq^n)^m q^{m+1\choose 2}}{(xsq^{n+1})_m}
\\
&=
\sum_{m,n \ge0 } \frac{(xs)^{m+n} q^{m+n+1\choose
    2}}{(xsq)_{m+n}}\frac{(-1)^n}{(sq)_{n}}
\\
&=
\sum_{N \ge0} \frac{(xs)^{N} q^{N+1\choose
    2}}{(xsq)_N}\sum_{n=0}^{N}\frac{(-1)^n}{(sq)_{n}}.
\end{align*}
We now return to~\eqref{B-itere}, where we set $s=1$. This gives
$$
\Ba(1)= 
\sum_{n \ge0} \frac{x^{n} q^{n+1\choose
    2}}{(xq)_n}\sum_{k=0}^{n}\frac{(-1)^k}{(q)_{k}}-
\Ba(1) \sum_{n\ge 0} \frac{(-x)^{n+1} q^{n+1\choose
    2}}{(q)_{n+1}(xq)_{n+1}}.
$$
Solving for $\Ba(1)$ gives
\beq\label{Ba-expr}
\Ba(1)= \frac {\KK(x)}{\JJ(x)},
\eeq
where {$\JJ$ and} $\KK$ are defined by~\eqref{JK-def}. Adding the
contribution~\eqref{Bna-sol} {of non-alternating fc
  elements} gives the second result of
{Theorem}~\ref{thm:ABD}.

\medskip

{We still have to explain the term $xsq T(sq)$ occurring
in~\eqref{Ba-eq}. 
Let $T(s)\equiv T(s;x,q)$ denote
  the \gf\ of alternating heaps such that $i$ is maximal (and having,
  as always in this subsection, at most one point in
  the last column ---  and hence exactly one). Then the heaps that we need to count to
  complete the proof of~\eqref{Ba-eq} ($i$ maximal, two points in the
  next-to-last column) are counted by $xsqT(sq)$, because they are
  obtained by adding one point in every column of a heap counted by
  $T(s)$, and then a final (rightmost) column containing one point
  (Figure~\ref{fig:T}, left). It remains to prove that $T(s)$
  satisfies~\eqref{eq-T}. This results again from a peeling procedure.
  The first term counts the heap reduced to one point, the second
  counts those such that the point in the last column is lower than any
  point in the next-to-last column, and the third one, as we have just
  explained, counts those in which the point in the last column is higher
  than one point in the next-to-last column. Another way to justify~\eqref{Tsol}
is to observe that heaps counted by $T(s)$  are just integer
partitions into distinct parts, where $q$ counts the weight, and $(xs)$
the size of the largest part. The integer $m$ in~\eqref{Tsol} gives the number of parts.}

\begin{figure}[ht]
\begin{center}
{\scalebox{1}{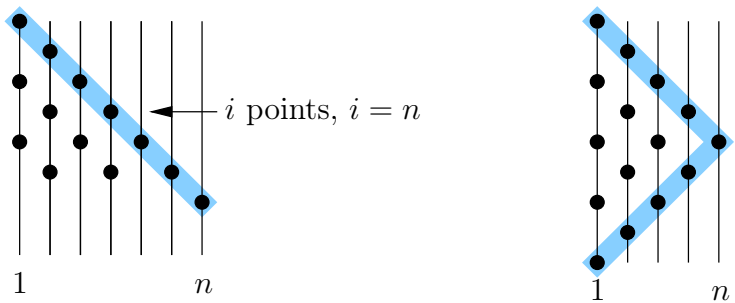}}
\caption{Left: An alternating heap where the parameter $i$ is
  maximal. The corresponding integer partition is $\mbmChange{\bjnChange{(7,4,2)}}$. Right:
  The self-dual case.}
\label{fig:T}
\end{center}
\end{figure}
 \qed

\medskip 
We now restrict this argument to involutions, that is, to self-dual
heaps.
The \gf\ of non-alternating fc involutions is
$\cBna=xq^2/(1-xq^2)\cA^{(1)}(x)$, where $\cA^{(1)}(x)$ counts self-dual
alternating heaps with one point in the first column and at most
one point in the last column. With our notation, $\cA^{(1)}=
\cA-1-x\cA$, and  by {the first result of} {Theorem}~\ref{thm:ABD-inv},
\begin{align}
\cBna&={ \frac{xq^2}{1-xq^2}\left( \cA(1-x)-1\right)}\label{cBna-0}
\\
&= \frac{xq^2}{1-xq^2}\left( \frac{(1-x) \JI(-xq)}{\JI(x)}-1\right).\label{cBna}
\end{align}
We now consider the \gf\ $\cBa$ of self-dual alternating heaps with at most one point in the
last column. We refine it as before into $\cBa(s)$, and adapt the argument that led to~\eqref{eq-A-inv} in
type  $A$. The first two cases contribute as before (with $\cA$ replaced by
$\cBa$). The third one splits again into two sub-cases: either $i$ is
maximal {(Figure~\ref{fig:T}, right)}, which gives the  \gf\
\beq\label{cT-def}
\cT(s)=\sum_{m\geq1}(xs)^mq^{m+1 \choose 2},
\eeq
 or deleting the rightmost diagonals
leaves a heap with a rightmost diagonal of size $j\ge i-1$, as in
Figure~\ref{fig:A-inv}. 
This  gives:
\beq\label{eq-cBa}
\cBa(s)=1+x\left(1+\frac{xsq}{1-sq^2}\right)\cBa(1)-\frac{x^2s^2q^3}{1-sq^2}\,
\cBa(sq^2)+\cT (s).
\eeq
By iteration, we obtain
$$
\cBa(s)=
\sum_{n\geq0}\frac{(-x^2s^2)^nq^{n(2n+1)}}
{\llp sq^2 \rrp_n}
\left(1+\cT(sq^{2n})+x\left(1+\frac{xsq^{2n+1}}{1-sq^{2n+2}}\right)\cBa(1)\right).
$$
  Setting $s=1$ and solving for $\cBa(1)$ gives
$$
\JI(x)\cBa(1)=
\sum_{n\geq0}\frac{(-x^2)^nq^{n(2n+1)}}
{\llp q^2 \rrp _n}
\left(1+\cT  (q^{2n})\right),
$$
{where $\JI(x)$ is defined by~\eqref{DD-def}.}
By definition~\eqref{cT-def} of the series $\cT(s)$, the right-hand
side can be written as  
$$
  \sum_{n,m\geq0}\frac{(-x^2)^nq^{n(2n+1)}}
{\llp q^2 \rrp_n}
x^mq^{2nm+{m+1\choose 2}}
=\sum_{N\geq0}x^Nq^{N+1\choose 2}\sum_{k=0}^{\lfloor
  N/2\rfloor}\frac{(-1)^k}
{\llp q^2 \rrp_k}\\
= \KI(x)
$$
where $\KI(x)$ is defined in~\eqref{UV-def} (we have set $N:=m+2n$ and $k:=n$ in the double sum). This finally gives 
\beq\label{cBalt}
\cBa(1)=
\frac{\KI(x)}{\JI(x)}.
\eeq
 Adding the
contribution~\eqref{cBna} {of non-alternating heaps} gives the second result of
{Theorem}~\ref{thm:ABD-inv}. \qed

\bigskip
\noindent{\bf Remark.} In the next subsection, we need to count
self-dual alternating heaps of type $B$ in which the first column contains an
odd number of points. Let us denote by $\cBa_{\mbox{\tiny odd}}(s)$
the corresponding \gf. The peeling argument that led
to~\eqref{eq-cBa} specializes into:
$$
\cBa_{\mbox{\tiny odd}}(s)=x\left(1+\frac{xsq}{1-sq^2}\right)\cBa_{\mbox{\tiny odd}}(1)-\frac{x^2s^2q^3}{1-sq^2}\,
\cBa_{\mbox{\tiny odd}}(sq^2)+\cT_{\mbox{\tiny odd}} (s),
$$
with 
$$
\cT_{\mbox{\tiny odd}} (s)=\sum_{m \ \hbox{\scriptsize
    odd}}(xs)^mq^{m+1\choose 2}.
$$
{Comparing with~\eqref{eq-cBa} shows that the only difference is the
replacement of $1+ \cT(s)$ by $\cT_{\mbox{\tiny odd}} (s)$.}
The iteration procedure yields:
\beq\label{CBaodd}
\cBa_{\mbox{\tiny odd}}{(1)}=
\frac{xq\KI_e(xq)}{\JI(x)},
\eeq
where $\KI_e(x)$ is the even part of $\KI(x)$ as defined by~\eqref{eop}.
\subsection{Type  $\boldsymbol D$}
\label{sec:peel-D}
 As explained in~\cite{BJN-long}, the structures of fc elements of type $B$
 and $D$ are closely related, and this results in similar
 \gfs.
 {In particular, we claim that the second equation of Proposition~4.6
 in~\cite{BJN-long} translates, in our notation, as}
\beq\label{D-sol}
D= 2\Ba -1-x A + \frac 1{xq} \Bna.
\eeq
Indeed, the definitions of the series $M^*$ and $Q^*$ given
in~\cite[Sec.~1.3]{BJN-long} in terms of certain paths, and the
correspondence between these paths and heaps, show that, with our
notation,  $M^*=1+xA$ and $Q^*=\Ba$. Proposition~4.6
 of~\cite{BJN-long} thus reads
 \begin{align*}
   B&= \Ba + \frac{x^2q^3}{1-xq^2}M^*(x) M(xq),\\
D&= 2\Ba-(1+xA) +  \frac{xq^2}{1-xq^2}M^*(x) M(xq).
 \end{align*}
The second term in the expression of $B$ must thus be $\Bna$, and
reporting this in the expression of~$D$ gives~\eqref{D-sol}.
 The third result of {Theorem}~\ref{thm:ABD} now follows
from~\eqref{Bna-sol},~\eqref{Ba-expr}, and our expression~\eqref{A-expr} of $A$.

\medskip
Following the analysis of~\cite[Prop.~3.1]{BJN-inv}, we claim that the
counterpart of~\eqref{D-sol} for fc \emm involutions, of type $D$ is:
\beq
\cD= 2\cBa_{\mbox{\tiny odd}} + 1+x \cA
 + \frac 1{xq} \cBna, \label{Dinv-sol}
\eeq
where $\cBa_{\mbox{\tiny odd}}$ counts self-dual alternating heaps of type $B$
with an odd number of points in the first
column. {Indeed, the definitions of the series $M$,  $Q$ and $Q^\circ$
  in terms of paths given in~\cite[Sec.~1.4 and Prop.~3.1]{BJN-inv}, and the
  correspondence between these paths and heaps, show that, with our
  notation, 
$$
\frac{M}{1-xM}= 1+x\cA, \quad  \frac{Q}{1-xM}=
\cBa \quad \hbox{and} \quad \frac{Q^\circ}{1-xM}=
\cBa_{\mbox{\tiny odd}}.
$$
With our notation, Proposition~3.1 of~\cite{BJN-inv} thus reads
\begin{align*}
  \cB&= \cBa+ \frac{x^2q^3}{1-xq^2}\frac{M(x)M(xq)}{1-xM(x)},
\\
\cD&= 2 \cBa_{\mbox{\tiny odd}} +1+x\cA + \frac{xq^2}{1-xq^2}\frac{M(x)M(xq)}{1-xM(x)}.
\end{align*}
The second term in the expression of $\cB$ must thus be $\cBna$, and
reporting this in the expression of~$\cD$ gives~\eqref{Dinv-sol}.}
The third result of {Theorem}~\ref{thm:ABD-inv} now follows
from~\eqref{CBaodd},~\eqref{cBna} and our expression~\eqref{cA-expr} of $\cA$. \qed

\subsection{New expressions for {some} path generating functions}
\label{sec:earlier}

In the previous subsection, we have used two identities between \gfs\
that came from~\cite{BJN-long} and~\cite{BJN-inv}, namely~\eqref{D-sol}
and~\eqref{Dinv-sol}, to express the series $D$ and $\cD$.
{In Section~\ref{sec:At}, we will use more of these identities
when counting fc elements of  type $\tA$}.  It is thus 
appropriate to give the explicit values of some generating functions for paths
appearing in~\cite{BJN-long} and~\cite{BJN-inv}, and characterized therein
by non-linear $q$-equations.

\medskip
We first express  the series involved in~\cite{BJN-long} in the
enumeration of fc elements of types $A$, $B$, and~$D$
 in terms of the series $\JJ$ and $\KK$
of~\eqref{JK-def}, and of the series $\HH$ of~\eqref{H-def}. 

\begin{Proposition}\label{prop:MQ}
  The series denoted $M(x)$ and $M^*(x)$
  in~\cite{BJN-long} admit the following explicit expressions:
$$
M(x)= 
\frac{\HH(xq)}{(1-x) \HH(x)}, \qquad \qquad M^*(x)=
\frac{\HH(xq)}{\JJ(x)}.
$$
 \end{Proposition}
\begin{proof}
 The series denoted $A^{FC}$ in~\cite{BJN-long} is, in our notation,
  $xA$. Hence Proposition~2.7 from~\cite{BJN-long} gives $M^*=1+xA$,
  and the above expression of $M^*$ follows from
{Theorem}~\ref{thm:ABD} and~\eqref{eq-HJ-alt}.

Let us now consider the series $M$. Equation~(4) in~\cite{BJN-long} gives
$$
M(x)= \frac{M^*(x)}{1-xM^*(x)},
$$
and the above expression of $M(x)$ follows from the expression of
$M^*(x)$ and~\eqref{eq-JH}. 
\end{proof}
\noindent{\bf Remark.} Our proof of {Theorem}~\ref{thm:ABD} only relies
on the description of fc elements in terms of heaps. An alternative
proof would have been to use the {characterization of the series $A$, $B$
and $D$  in terms of non-linear $q$-equations} given in~\cite{BJN-long},
and to check that our expressions in terms of $J$ and $K$ satisfy
these equations.  One advantage of our recursive approach is that
it \emm constructs, the
series $J$ and $K$, while one would not understand where they come
from  if we had simply checked the systems of~\cite{BJN-long}. In the
next section however, we solve the $\tA$-case {by checking
  the relevant $q$-equations} of~\cite{BJN-long}. 
\medskip

To conclude this section, we state the counterpart of
Proposition~\ref{prop:MQ} for involutions. Recall that the series
$\HI$ 
is defined {by combining~\eqref{DD-def} and~\eqref{eop}}. 

\begin{Proposition}\label{prop:Dyckwalks}
Let us denote by $\cM(x)\equiv \cM(x,q)$ the series denoted by
$M(x)$ in~\cite{BJN-inv}.  Then
$$
\cM(x)=\frac{\HI(xq)}{\HI(x)}.
$$
\end{Proposition}

\begin{proof}
The series denoted $\bar A$ in~\cite{BJN-inv} is our series
$x\cA$. Hence Proposition~3.1 in~\cite{BJN-inv} reads
$$ 
\frac{\cM(x)}{1-x\cM(x)}=1+x\cA.
$$
{Combined with the expression~\eqref{cA-expr} of $\cA$, this gives:
$$
\cM(x)= \frac{\JI(x)+x\JI(-xq)}{(1+x)\JI(x)+x^2\JI(-xq)}.
$$
Using the explicit expression of $\JI$, one checks that the numerator
(resp. denominator) of this expression is $\HI(xq)$ (resp. $\HI(x)$).}
\end{proof}

\section{Affine type $\boldsymbol{\tA}$}
\label{sec:At}
In this section, we first establish the simple expression of $\tA$  announced in
{Theorem}~\ref{thm:A-tA}. 
{As proved in~\cite[Prop.~3.3]{St1}, fc elements in $\tA_{n-1}$
are in one-to-one correspondence with  alternating heaps over the
$n$-point cycle (Figure~\ref{fig:Alter}, right).}
In~\cite{BJN-long},
the series $\tA$ was characterized  in terms of two series $\vO$ and
$\vO^*$ related to the series $M$ and $M^*$ of
Proposition~\ref{prop:MQ}. This yields our
nice expression of $\tA$ in terms of $J$ and its derivative.  We then
proceed similarly with  the \gf\ $\ctA$ of fc involutions in type~$\tA$ {({Theorem}~\ref{thm:A-tA-inv})}.

This computational  approach does not explain the simplicity
of $\tA$ and $\ctA$. In {the extended abstract}~\cite{BBJN-bij}, we provide
bijective explanations of them.

\subsection{All fully commutative elements}

\begin{Theorem}\label{thm:O}
  With $\JJ(x)$ and $\HH(x)$ defined by~\eqref{JK-def} and~\eqref{H-def}
  above, {the} \gfs\ $\vO(x)$ and $\vO^*(x)$
defined  in Corollary 2.4 of~\cite{BJN-long} are given by
\beq\label{expr-Oc}
\vO(x)= \frac 1 {1-x} -x \frac{H'(x)}{\HH(x)} + xq
\frac{H'(xq)}{\HH(xq)},
\eeq
\beq\label{expr-Ocs}
  \vO^*(x)= 1- x \frac{J'(x)}{\JJ(x)}+xq \frac{H'(xq)}{\HH(xq)}.
\eeq%
The \gf\ of fully commutative elements in type $\tA$ is
$$
 \tA:=
\sum_{n\ge 1} \tA_{n-1}^{FC}(q) x^n=- x \frac{\JJ'(x)}{\JJ(x)}- \sum_{n\ge 1}
 \frac{x^n  q^n}{1-q^n}.
$$
\end{Theorem}
\noindent{\bf Remarks} \\
1. Above, we have taken $\tA_1= A_2$ so that
$\tA_1^{FC}(q)=(1+q)/(1-q)$. The coefficient  $\tA_0^{FC}(q)$ of
$x^1$ in $\tA(x,q)$ is irrelevant ({and equal to 1}).

\smallskip
\noindent 2.  {Fully commutative elements of $\tA$ having full support
  (that is, in which each generator occurs) can be seen as periodic
  staircase polyominoes. These objects have recently been studied
  for their own sake~\cite{aval-ppp,boussicault-laborde}, and in particular
  counted according to their perimeter. \mbmChangenew{When proving the
    above theorem, we will obtain  their length generating function as
  \beq\label{ppp}
  - xq \frac{\HH'(xq)}{\HH(xq)}- \sum_{n\ge 1}
 \frac{x^n  q^n}{1-q^n}.
  \eeq
   This formula can be explained combinatorially by the methods of~\cite{BBJN-bij}.}
  }

\begin{proof}
 Both series $\vO$ and $\vO^*$ are expressed in terms of $M$ and $M^*$ in Eq.~(2)
 of~\cite{BJN-long}. For instance,
\beq\label{vO-def}
\vO(x)= M(x) \left(1+ x^2q \frac{\partial (xM)}{\partial x}(xq)\right).
\eeq
The notation should be understood as follows: one first takes the
derivative of $xM$, and then evaluates it at $xq$.  Using the expression of $M$  in terms of $H$
(Proposition~\ref{prop:MQ}), the right-hand side can be written in terms of
$H(x)$, $H(xq)$, $H(xq^2)$, $H'(xq)$ and $H'(xq^2)$. Now using the
linear $q$-equation~\eqref{eqH-lin} satisfied by $H$, we can  express $H(xq^2)$ and
$H'(xq^2)$ in terms of $H(x)$, $H(xq)$, $H'(x)$ and $H'(xq)$. This
yields the expression of $\vO$ given in the {theorem}. 

The expression of $\vO^*$ can be proved similarly, starting from
Eq.~(2) in ~\cite{BJN-long}. Alternatively, we
can  derive it from our
expression of $\vO$ since this equation implies that
$\vO^*=\vO M^*/M= \vO (1-x)H/J$. One then checks that this coincides
with~\eqref{expr-Ocs} using the expression~\eqref{eq-JH} of
$J$ in terms of $H$.

Consider now the series $\tA$. The second part of Eq.~(1)
from~\cite{BJN-long} tells us that
\beq\label{AO}
\tA(x)-\tA(xq)= \vO(xq)-1 -2\,\frac{xq}{1-xq} +\vO^*(x) -\vO^*(xq).
\eeq
(We have used the fact that $\vO(x)$ has constant term $1$.) We now
use our expression of $\vO$:
$$
\tA(x)-\tA(xq)= -\frac{xq}{1-xq} -xq\frac{H'(xq)}{H(xq)}
+xq^2\frac{H'(xq^2)}{H(xq^2)} +\vO^*(x) -\vO^*(xq).
$$
Since
$\tA(x)$ has constant term $0$, and $\vO^*(x)$ has constant term $1$,
iterating this equation gives: 
$$
\tA(x)= -\sum_{i\ge 1}
\frac{xq^i}{1-xq^i}-xq\frac{H'(xq)}{H(xq)}+\vO^*(x)-1,
$$
and the last statement of the {theorem} follows using our expression
of $\vO^*(x)$. (We also expand the sum over $i$ in powers of $x$.)

\mbmChangenew{If we only want to count fc elements with full support,
  then the arguments of~\cite[Cor.~2.4]{BJN-long} shows that we only
  have to drop the terms $\vO^*(x)$ and $\vO^*(xq)$ in~\eqref{AO}. The
  expression~\eqref{ppp} of their \gf \ follows.}
 \end{proof}

\subsection{Fully commutative involutions}\label{sec:typeAtilde}
\begin{Theorem}\label{thm:At-inv}
 Let $\cM(x)$ be given by Proposition~\ref{prop:Dyckwalks}. {Let
 $\vcO(x)=\sum_{n\ge 0} \vO_n(q)x^n$, where the
 polynomials  $\vO_n(q)$ are defined} in Proposition~3.3 of~\cite{BJN-inv},
 and let
\beq\label{vcOs-def}
\vcO^*(x)=\frac{\cM(x)}{1-x\cM(x)} \left(1+ x^2q\, \frac{\partial (x
    \cM)}{\partial x}(xq)\right).
\eeq
Then
\beq\label{vcO}
\vcO(x)=1-x\frac{\HI'(x)}{\HI(x)} +xq\frac{\HI'(xq)}{\HI(xq)},
\eeq
\beq\label{vcOs}
\vcO^*(x)= 1-x\frac{\JI'(x)}{\JI(x)} +xq\frac{\HI'(xq)}{\HI(xq)},
\eeq
   where $\JI(x)$ is defined by~\eqref{UV-def} and $\JI_e(x)$ denotes
   its even part in $x$ (see~\eqref{eop}).

The \gf\ of fully commutative involutions in type $\tA$ is
$$
\ctA=
\sum_{n\ge 1} \ctA_{n-1}^{FC}(q) x^n=
-x \frac{\JI'(x)}{\JI                  (x)}.
$$
\end{Theorem}
\noindent{\bf Remark.} The coefficients of $x^1$ and $x^2$ in $\ctA$
are irrelevant.
\begin{proof}
 {{The polynomials
  $\vO_n(q)$ and the series $\cM(x)$ are defined in terms of Dyck-like
  paths in~\cite{BJN-inv} (see Proposition~3.3 and Section 1.4, \mbmChangenew{where the $\check{F}$ notation
  is explained})}. These definitions lead to:}  
$$
\vcO(x)= \cM(x) \left(1+ x^2q\, \frac{\partial (x
    \cM)}{\partial x}(xq)\right).
$$
 We can now derive the
  expression~\eqref{vcO} exactly  as we proved~\eqref{expr-Oc}
  from~\eqref{vO-def}, {using Proposition~\ref{prop:Dyckwalks} for the
    expression of $\cM$ and the linear $q$-equation~\eqref{HI-eq} satisfied
    by $\HI$.}

 The definition~\eqref{vcOs-def} of $\vcO^*(x)$ now reads $\vcO^*(x)= \vcO(x)/(1-x
 \cM(x))$. Using the expression of $\cM$ given in
 Proposition~\ref{prop:Dyckwalks}, and the expression~\eqref{JI-JIe} of
 $\JI$ in terms of $\JI_e$, one readily checks that this coincides with our expression of $\vcO^*(x)$.

\medskip
Consider now the series $\ctA$. Proposition~3.3 in~\cite{BJN-inv}
tells us that
\begin{align*}
  \ctA(x)-\ctA(xq)&= \vcO(xq)-1 +\vcO^*(x)-\vcO^*(xq)\\
&= -xq\ \frac{\HI'(xq)}{\HI(xq)} +xq^2\, \frac{\HI'(xq^2)}{\HI(xq^2)} +
\vcO^*(x)-\vcO^*(xq) 
\end{align*}
by~\eqref{vcO}. Iterating this formula gives
$$
\ctA(x)= -xq\ \frac{\HI'(xq)}{\HI(xq)} +\vcO^*(x)-1 = -x\,
\frac{\JI'(x)}{\JI (x)}
$$
by~\eqref{vcOs}.
\end{proof}

\begin{Corollary}
  The number of fully commutative involutions in $\tA_{2m}$ is finite,
  equal to $4^m$.
\end{Corollary}
\begin{proof}
  We will prove this using the enumerative results of
{Theorem}~\ref{thm:At-inv}, though there are more direct
  combinatorial arguments based on the structure of these involutions, {see~\cite{BJN-inv}}.

By~\cite[Eq.~(2)]{BJN-inv}, the series $\cM(x)$ satisfies
$$
\cM(x)=1+x^2q\cM(x)\cM(xq).
$$
Hence it is finite at $q=1$, equal to $(1-\sqrt{1-4x^2})/(2x^2)$. Then
it follows from~\eqref{vcOs-def} that at $q=1$,
\beq\label{vcOs-q1}
\vcO^*(x)= \frac 1 2 \left( \frac 1{\sqrt{1-4x^2}}+\frac 1
  {1-2x}\right).
\eeq
We want to determine  the series
$$
\lim_{q\rightarrow 1} \sum_{m\ge 0} \tilde{\cA}^{FC}_{2m}(q)x^{2m+1},
$$
which is the odd part of $\tilde{\cA}(x,q)$, in the limit
$q\rightarrow 1$. 
{Let us compare the expressions of $\ctA$ and $\vcO^*$ in
{Theorem}~\ref{thm:At-inv}: the first and third terms in the
  expression of $\vcO^*$ are even in $x$ (because of the derivative),
  and thus the series we want to compute  is also the odd part of $\vcO^*(x)$}, that is (thanks
to~\eqref{vcOs-q1}), $x/(1-4x^2)$. The result follows.
\end{proof}

\section{Another recursive approach: Peeling a column}
\label{sec:column}
In order to count fc elements in infinite types $\tB$, $\tC$ and
$\tD$, we need to count  alternating
heaps over a path, with no restriction on the number of points in
the first and last columns. 
The recursive approach of the previous section applies, 
but the resulting expressions are complicated and we do not give
them.

Instead, we describe in this section another recursive
approach  to count alternating heaps over a path. The principle
simply consists in deleting the first column of the heap. This is essentially the idea behind the encoding of alternating heaps as paths in~\cite{BJN-long} but we will not use this encoding here.
By forcing the first  and/or last column
to contain at most one point, we  obtain new expressions for the
series $A$ and $\Ba$. Remarkably, these expressions involve negative
powers of the length variable $q$. This phenomenon has already been
observed in polyomino enumeration~(see
eg. \cite{feretic2,feretic} or~\cite[Ex.~5.5.2]{goulden-jackson}).

\subsection{Alternating heaps over a path}
We begin with a new expression for the series $A$.

\begin{Theorem}\label{thm:A-neg}
  The \gf\ $A(x,q)\equiv A$ of fully commutative  elements of type~$A$,
  defined by
$$
A=\sum_{n\ge 0} A_n^{FC}(q) x^n
$$
{and given by~\eqref{A-expr}}, can also be written as
$$
A=\frac{ \sum_{k\ge 0}x^k \sum_{i=0}^k
  {k \qchoose i} {k+1 \qchoose i+1}q^{-i(k-i)}}
{1+ \sum_{k\ge 1}x^k \sum_{i=0}^{k-1}
  {k-1\qchoose i}{k \qchoose i}q^{-i(k-i)}}
$$
where ${k \qchoose i}$ is the $q$-binomial coefficient:
$$
{k \qchoose i}= \frac{(q)_k}{(q)_i(q)_{k-i}}.
$$
\end{Theorem}

The next \mbmChangenew{theorem} gives a new expression for alternating heaps of
type $B$ (so far expressed by~\eqref{Ba-expr}) and an expression for 
alternating heaps over a path. {By Theorem~3.4
  of~\cite{BJN-long}, these heaps form a subset of fc elements in
  type $\tC$.} As before, the variable $x$ records
the number of vertices of the path (that is, the number of generators in the
group), and $q$ the number of points in the heap (the length of the
corresponding fc element). {By an \emm empty, path, we mean
  the $n$-point path when $n=0$.}

\begin{Theorem}\label{thm:BL-neg}
 Let $N(x)$ and $S(x)$ be the following series:
$$
N(x)= \sum_{k\ge 1} x^k q^{-{k\choose 2}}\frac{(-q)^2_{k-1}}{1-q^k},
$$
and
\beq\label{S-def}
S(x)= N(x/q)-N(xq)= \sum_{k\ge 1} {x^k q^{-{k+1\choose 2}}(-q)_{k-1}(-q)_k}.
\eeq
   The \gf\ of  alternating  heaps of type $B$,  expressed
   in~\eqref{Ba-expr}, can also be written as
\beq\label{B-neg}
\Ba= 1+ S(xq) -  {xS(x)}A,
\eeq
where $A$ is the \gf\ of fc elements of type $A$, expressed in
{Theorem}~\ref{thm:ABD} or alternatively in {Theorem}~\ref{thm:A-neg}.

The \gf\ of  alternating heaps over a non-empty path  is
$$
\AL(x)= N(x)-xS(x)\Ba,
$$
where $\Ba$ is the \gf\ of alternating heaps of type $B$, expressed
by~\eqref{Ba-expr} or alternatively by~\eqref{B-neg}. 
\end{Theorem}
\noindent{{\bf Remark.} As noted already, the expressions of
  $A, B$ and $D$ given in  {Theorem}~\ref{thm:ABD} define series in $x$
  with rational coefficients in $q$, with poles at roots of unity ---
  even though we know from combinatorial reasons that these
  coefficients must be polynomials in $q$. In particular, it is not
  immediate to derive from {Theorem}~\ref{thm:ABD} the value at
  $q=1$ of these three series. Now with {Theorem}~\ref{thm:A-neg}, we have
 an expression of $A$ as a series in $x$ whose coefficients
  are Laurent polynomials in $q$, and the specialization $q=1$ is
  straightforward: the numerator of $A$ becomes
$$
\sum_{k\ge0} x^k {2k+1 \choose k}= \frac 1 {2x} \left( \frac 1
  {\sqrt{1-4x}}-1\right),
$$
and its denominator
$$
1+\sum_{k\ge 1} x^k {2k-1\choose k}= \frac 1 {2} \left( \frac 1
  {\sqrt{1-4x}}+1\right),
$$
so that we recover the Catalan \gf\ for $A(x,1)$:
$$
A(x,1)= \sum_{n\ge 0} A_n^{FC}(1) x^n= \frac{1-2x-\sqrt{1-4x}}{2x^2}
=\sum_{n\ge 0} \frac 1 {n+2}{2n+2\choose n+1}x^n.
$$
Similarly,~\eqref{B-neg} expresses $\Ba$ as a series in $x$ whose
coefficients are Laurent polynomials in $q$. From this and~\eqref{Bna-A}, one can derive
a similar expression for $B=\Ba+\Bna$, and finally for $D$ using~\eqref{D-sol}. In
these expressions, one can set $q=1$ to recover Stembridge's
enumeration of fc elements in $B_n$ and $D_{n+1}$ (see~\cite{St3}). For
the series $L(x)$ counting alternating heaps, we obtain the following periodicity result.
\begin{Corollary}\label{cor:L}
  Write $L(x)=\sum_{n\ge 1} L_n x^n$, where $L_n\equiv L_n(q)$  counts
  alternating heaps over the $n$-point path. Then $L_n$ is a rational
  fraction in $q$
  of denominator $(1-q^n)$. Consequently, the sequence of its
  coefficients is ultimately periodic, with period dividing $n$.
\end{Corollary}
\begin{proof}
  The expression of $L(x)$ given in
  \mbmChangenew{Theorem}~\ref{thm:BL-neg} shows that $L_n$ is
  a fraction in $q$ with denominator $q^e (1-q^n)$, for some
  integer $e$ (indeed, this holds in $N(x)$, and the series $S$
  and $\Ba$ have polynomial coefficients). But then the combinatorial
 description of $L_n$ shows that $q=0$ cannot be a pole.
\end{proof}
}
\begin{proof}[Proof of {Theorems~\ref{thm:A-neg} and~\ref{thm:BL-neg}}]
    Let us begin with alternating heaps of type $B$, {excluding the
    case $B_0$}, and
    denote their \gf\ by $\tBa:=\Ba-1$. We enrich this series
  by taking into account  (with a new variable $u$) the number of occurrences of the ``leftmost'' generator,
  denoted~$t$ in Figure~\ref{fig:dynkin}. We denote by $n$ the number of
  generators, and partition
  the set of such   heaps into four classes: 
  \begin{itemize}
  \item if the first column is empty, $t$ does not occur. By
    Definition~\ref{def:altheaps} of alternating heaps, the second
    column contains at most one point, so that deleting the first
    column leaves  an alternating
    heap of type $A_{n-1}$; the \gf\ for this class is thus $xA$;
\item from now on $t$ occurs. If $n=1$ then the only element that contributes is $t$, with
  \gf\ $xuq$;
\item from now on $n\ge 2$, and $t$ occurs. If the second column is
  empty, that is, the generator $s_1$
  does not occur, then there is only one occurrence of $t$, followed
  by an empty column and an alternating heap of type $A_{n-2}$.  The
  \gf\ is thus $x^2uq A$; 
\item otherwise, deleting every point of the first column leaves an
  alternating heap of type $B_{{n-1}}$ with a non-empty first
  column. The \gf\ of such heaps is $\tBa(u)-xA$. Conversely, an alternating heap of type $B_{n-1}$ with $k$ points in the
  first column gives rise to four alternating heaps by adding a column: two
  with $k$ points in the first column, one with $k-1$ points and one with
  $k+1$ points (Figure~\ref{fig:B-alt}). However, if   $k=1$,  the choice of
  $k-1$ occurrences of
  $t$ gives  a heap with no occurrence of $t$, which we
  have already counted in the first class. Hence the contribution of this last  class is:
$$
x\left(\frac1{uq} +2 +uq\right) \left( \tBa(uq)-xA\right) -x (A-1-xA),
$$
since $A-1-xA$ counts alternating heaps of type $A$ with
exactly one point in the first column.
  \end{itemize}

\begin{figure}[ht]
\begin{center}
\includegraphics[scale=0.9]{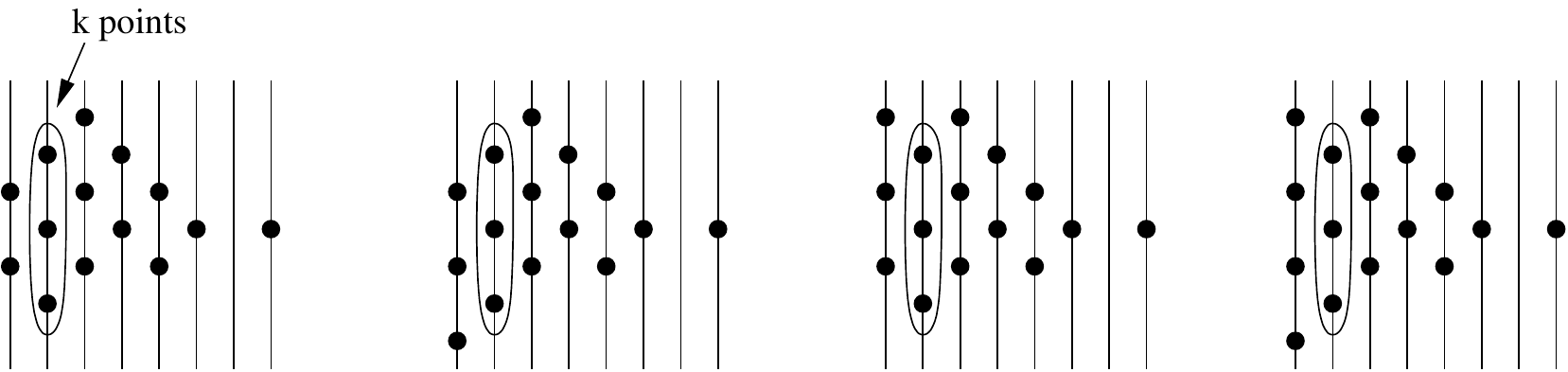}
\caption{{Adding} a column to the left of  an alternating heap of type $B$.} 
\label{fig:B-alt}
\end{center}
\end{figure}

\noindent Putting  together these four cases gives
\beq\label{eq-tB}
\tBa(u)= x(1+uq) - \frac{x^2(1+uq)}{uq} A +
\frac{x(1+uq)^2}{uq}\tBa(uq).
\eeq
This equation can be solved using the iteration method of Section~\ref{sec:peel-A}:
\begin{align}
  \tBa(u)&= \sum_{k\ge 1} x^{k-1}u^{-k+1}q^{-{k\choose
    2}}(-uq)_{k-1}^2\left( 
x(1+uq^k)-\frac{x^2(1+uq^k)}{uq^k}A\right) \nonumber\\
&= \sum_{k\ge 1} x^k u^{-k+1} q^{-{k\choose 2}}
  (-uq)_{k-1}(-uq)_k
-A \sum_{k\ge 1} x^{k+1} u^{-k} q^{-{k+1\choose 2}}
  (-uq)_{k-1}(-uq)_k.\label{Balt-u}
\end{align}
Setting $u=1$ and adding the contribution of the empty heap gives~\eqref{B-neg}.

\medskip
We can now derive the expression of $A$ given in
{Theorem}~\ref{thm:A-neg}. The series  $\tBa(0)$ counts heaps with an empty first
column, and  thus coincides with the series $xA$. Let us
generalize~\eqref{S-def} by denoting
$$
S(x,u)=\sum_{k\ge 1} x^{k} u^{-k} q^{-{k+1\choose 2}}  (-uq)_{k-1}(-uq)_k,
$$
so that~\eqref{Balt-u} reads
$$
\tBa(u)= uS(xq,u)-xA\, S(x,u).
$$
 Extracting from this the coefficient of $u^0$  gives
$$
xA=\tBa(0)={[u^{-1}]S(xq,u) } {-}  xA [u^0] S(x,u),
$$
and hence
$$
xA=\tBa(0)=\frac{[u^{-1}]S(xq,u) }{1+ [u^0] S(x,u)}.
$$
Performing the coefficient extraction explicitly gives
{Theorem}~\ref{thm:A-neg}, using
$$
(-uq)_k=\sum_{i=0}^k u^i q^{i+1\choose 2} {k\qchoose i}.
$$

 \medskip

  We now {apply}
 column peeling to general
  alternating heaps over the $n$-point path, counted by the series
  $L$. There are no more constraints on the first and last
  columns. As before, an additional variable $u$
  records the number of points in the first column. The four cases
  listed in the derivation of~\eqref{eq-tB} are transformed as follows:
 \begin{itemize}
  \item if the first column is empty, then the second must contain at
    most one point. Thus what remains after deleting the first column
    is an alternating heap of type $B_{{n-1}}$, reflected in a vertical
    line. The \gf\ for this first case is therefore $x\Ba$;
\item from now on the first column is non-empty. If $n=1$ then all heaps are
  alternating, with   \gf\  $xuq/(1-uq)$;
\item assume the second column is empty. Then there is  just one point in
  the first column, followed by an empty column, and then a
  reflected alternating heap of type $B_{{n-2}}$.  The
  \gf\ is $x^2uq \Ba$; 
\item finally, if there are $k$ points in the second column, with $k\ge 1$, then
  we can have $k-1$, $k$ or $k+1$ points in the first column. As
  before, the case where
  $k=1$ and the first column is empty has  already been counted. Hence
  the contribution of this final case is:
$$
x\left(\frac1{uq} +2 +uq\right) \left( \AL(uq)-x\Ba\right) -x (\Ba-1-x\Ba),
$$
since $\Ba-1-x\Ba$ counts alternating heaps of type $B$, reflected in
a vertical line, and having exactly one point in the first column.
 \end{itemize}

\noindent Putting  together these four cases gives
$$ 
\AL(u)= \frac x{1-uq} - \frac{x^2(1+uq)}{uq} \Ba +
\frac{x(1+uq)^2}{uq}\AL(uq).
$$ 
Iterating  yields
$$ 
\AL(u)= \sum_{k\ge 1} x^k u^{-k+1} q^{-{k\choose 2}}
 \frac{{(-uq)_{k-1}}^2}{1-uq^k}
-\Ba \sum_{k\ge 1} x^{k+1} u^{-k} q^{-{k+1\choose 2}}
  (-uq)_{k-1}(-uq)_k.
$$ 
Setting $u=1$ gives the last result of \mbmChangenew{Theorem}~\ref{thm:BL-neg}.
\end{proof}

\subsection{Self-dual alternating heaps over a path}

We now restrict the enumeration to self-dual heaps.
We begin with a new expression for the series ${\mathcal A}$, already
expressed in {Theorem}~\ref{thm:ABD-inv}.

\begin{Theorem}\label{thm:newinvolA}
The generating function ${\mathcal A}(x,q)\equiv{\mathcal A}$ of fully
commutative involutions of type $A$, defined by 
$$
\cA=\sum_{n\geq0}{\mathcal A}_n^{FC}(q)x^n
$$
{and given by~\eqref{cA-expr}}, can also be written as:
$$
\cA=\frac{\sum_{j\geq0}x^j
{j \qschoose \lfloor   j/2\rfloor}
\,q^{-\lfloor
    j/2\rfloor^2+\chi(j\,\mbox{odd})}}{1+\sum_{j\geq1}(-x)^j
{j-1 \qschoose \lfloor    j/2\rfloor}
\,q^{-\lfloor j/2\rfloor^2}}
$$ 
where $k \qschoose i$ is the $q^2$-binomial coefficient:
$$
{k \qschoose i}= \frac{\llp q^2\rrp _k}{\llp q^2 \rrp_i \llp
  q^2\rrp_{k-i}}
= \frac{(q^2;q^2)_k}{(q^2;q^2)_i(q^2;q^2)_{k-i}}.
$$
\end{Theorem}

The next \mbmChangenew{theorem} gives a new expression for self-dual alternating
heaps of type $B$ and an expression for self-dual alternating  heaps
over a path. {These heaps  form a subclass of fc involutions in
  type $\tC$.}

\begin{Theorem}\label{prop:newinvolBetsegment}
Let $\mathcal{N}(x)$ and  $\mathcal{S}(x)$ be the following series:
\beq\label{cN-def}
\mathcal{N}(x)=\sum_{k\geq1}x^kq^{-\bi{k}{2}}\frac
{\llp -q^2 \rrp_{k-1}}
{1-q^k},
\eeq
\beq\label{cS-def}
\mathcal{S}(x)=\mathcal{N}(x/q)-\mathcal{N}(x)
=\sum_{k\geq1}x^kq^{-\bi{k+1}{2}}
{\llp -q^2 \rrp_{k-1}}.
\eeq
The generating function of self-dual alternating heaps of type $B$,
already given by~\eqref{cBalt}, can also be written as
\beq\label{cBa-new}
\cBa 
=1+ \cS(xq)+\cS(xq^2) -x \left( \cS(x)-\cS(xq)\right) \cA,
\eeq
where $\mathcal{A}$ is the generating function of 
 fully commutative involutions of
type $A$, given in {Theorem}~\ref{thm:ABD-inv} or
alternatively in {Theorem}~\ref{thm:newinvolA}. 

The generating function for  self-dual alternating heaps
over a  {non-empty} path is  
$$
\mathcal{L}=\mathcal{N}(x)-x(\mathcal{S}(x)-\mathcal{S}(xq))\cBa,
$$
where $\cBa$
 is the generating function 
of  self-dual
alternating heaps of type $B$, given by~\eqref{cBalt} or
alternatively by~\eqref{cBa-new}. 
\end{Theorem}
\begin{proof}[Proof of {Theorems~\ref{thm:newinvolA}
  and~\ref{prop:newinvolBetsegment}}]
We begin with self-dual alternating heaps of type~$B$, {excluding again
the case $B_0$}, and
denote their generating function by
$\hat{\mathcal{B}}^{\rm (a)}
:=\cBa-1$. 
We count them by
specializing to self-dual heaps the argument that led to~\eqref{eq-tB} for
alternating heaps of type $B$. The first three cases described there
now contribute $x \cA$, $xuq$ and $x^2uq \cA$ respectively. In the
fourth (and last) case, we observe that, in a self-dual heap, the
first two columns cannot both contain  $k$ points.
Hence the contribution of the fourth case is now
$$
x\left(\frac{1}{uq}+uq\right)\left(\hat{\mathcal{B}}^{\rm(a)}(uq)-x\mathcal{A}\right)-x(\mathcal{A}-1-x\mathcal{A}).
$$
Putting together the four cases gives
$$
\hat{\mathcal{B}}^{\rm(a)}(u)=x(1+uq)-\frac{x^2}{uq}(1-uq)\mathcal{A}
+\frac{x}{uq}(1+u^2q^2)\,\hat{\mathcal{B}}^{\rm(a)}(uq).
$$
 By iterating,  we obtain:
\begin{multline}
\label{altBversion2enu}
\hat{\mathcal{B}}^{\rm(a)}(u)=\sum_{k\geq1}x^ku^{1-k}q^{-\bi{k}{2}}(1+uq^k)
\llp -u^2q^2 \rrp_{k-1}\\
-\mathcal{A}\sum_{k\geq1}x^{k+1}u^{-k}q^{-\bi{k+1}{2}}(1-uq^k)
\llp -u^2q^2 \rrp_{k-1}.
\end{multline}
Setting $u=1$ and adding the contribution of the empty heap gives the
expression~\eqref{cBa-new} of $\cBa$.

\medskip
We can now  derive the alternative expression of $\cA$ given in
{Theorem}~\ref{thm:newinvolA}. The series $\hat{\mathcal{B}}^{\rm(a)}(0)$
counts heaps with an empty first column, and thus coincides with the
series $x\mathcal{A}$. Let us generalize~\eqref{cS-def} by denoting
$$
\mathcal{S}(x,u)=\sum_{k\geq1}x^ku^{-k}q^{-\bi{k+1}{2}}
 \llp -u^2q^2 \rrp _{k-1},
$$
so that~\eqref{altBversion2enu} reads 
\begin{equation}\label{Bachapeauenu}
\hat{\mathcal{B}}^{\rm(a)}(u)=u\mathcal{S}(xq,u)+u^2\mathcal{S}(xq^2,u)
-x\mathcal{A}\left( \mathcal{S}(x,u)-u\mathcal{S}(xq,u)\right).
\end{equation}
Extracting the coefficient of $u^0$ gives
$$
x\mathcal{A}=\hat{\mathcal{B}}^{\rm(a)}(0)=\frac{[u^{-1}]\mathcal{S}(xq,u)+[u^{-2}]\mathcal{S}(xq^2,u)}{1+[u^0]\mathcal{S}(x,u)-[u^{-1}]\mathcal{S}(xq,u)}. 
$$
The coefficient extraction is performed explicitly thanks to
$$
\llp -u^2q^2 \rrp_k= \sum_{i=0}^k u^{2i} q^{i(i+1)} {k \qschoose i},
$$
and one thus obtains {Theorem}~\ref{thm:newinvolA}. 

\medskip
Our final step is to count self-dual alternating heaps, by
specializing to the self-dual case the argument that led
to~\eqref{eq-tB} for general alternating heaps. The first three cases
described there  now contribute $x \cBa$, $xqu/(1-uq)$ and $x^2uq
\cBa$, respectively. In the fourth case, we must exclude heaps in
which the first two columns would have the same size. Hence the
contribution of the fourth case is now:
$$
x\left(\frac{1}{uq}+uq\right)\left(\mathcal{L}(uq)-x\mathcal{B}^{\rm(a)}\right)
-x\left(\mathcal{B}^{\rm(a)}-1-x\mathcal{B}^{\rm(a)}\right). 
$$
Putting together the four cases gives
\beq\label{cL-eq}
\mathcal{L}(u)=\frac{x}{1-uq}-\frac{x^2}{uq}(1-uq)\mathcal{B}^{\rm(a)}+\frac{x}{uq}(1+u^2q^2)\mathcal{L}(uq).
\eeq
Iterating yields
\begin{align}
\mathcal{L}(u)&=\sum_{k\geq1}x^ku^{1-k}q^{-\bi{k}{2}}\frac{
\llp -u^2q^2 \rrp_{k-1}
}{1-uq^k}
-\mathcal{B}^{\rm(a)}\sum_{k\geq1}x^{k+1}u^{-k}q^{-\bi{k+1}{2}}(1-uq^k)
\llp -u^2q^2 \rrp_{k-1}, \nonumber
\\
&=\sum_{k\geq1}x^ku^{1-k}q^{-\bi{k}{2}}\frac{\llp -u^2q^2 \rrp_{k-1}}{1-uq^k}
-x\left( \cS(x,u)-u\cS(xq,u)\right)\mathcal{B}^{\rm(a)}. \label{altSegmentversion2enu}
\end{align}
Setting $u=1$ gives the second result of \mbmChangenew{Theorem}~\ref{prop:newinvolBetsegment}.
\end{proof}

Later we will need the following result,
 which gives the generating
functions of self-dual alternating heaps over a path, with
 parity constraints on the first and/or last columns. Recall
 the notation  $F_e(x)$ and  $F_o(x)$ for the even and odd parts (in
 $x$) of a
 series $F(x)$ (see~\eqref{eop}).

\begin{Proposition}\label{prop:newinvolBetsegmentOdd}
The generating function of self-dual alternating heaps of type $B$
having an odd number of points in the first column, given
by~\eqref{CBaodd}, admits the following alternative expression:
\beq\label{CBaodd-new}
\cBa_{\mbox{\tiny \rm odd}}=\cS_e(xq)+\cS_o(xq^2)-x \left(
  \cS_o(x)-\cS_e(xq)\right) \cA,
\eeq
where $\cS$ is defined by~\eqref{cS-def} and $\cA$ counts fc
involutions of type $A$ ({Theorems}~\ref{thm:ABD} and~\ref{thm:newinvolA}).

The generating function of self-dual alternating heaps over a path,
having an odd number of points in the first column, is
$$
\mathcal{L}_{\mbox{\tiny \rm odd}}=\frac{1}{2}\left(\mathcal{N}(x)+\tilde
  \cN(-x)\right)
- x \left(
  \cS_o(x)-\cS_e(xq)\right) \cBa,
$$
where $\cN$ is defined by~\eqref{cN-def}, $\cBa$ counts
{self-dual} alternating heaps of type~$B$ (see~\eqref{cBalt} and~\eqref{cBa-new}), and
$$
\tilde \cN
(x)=\sum_{k\geq1}x^kq^{-\bi{k}{2}}\frac{
\llp -q^2 \rrp_{k-1}
}{1+q^{k}}.
$$

The generating function of self-dual alternating heaps over a path
having an odd number of points in the first and  last
columns is  
$$ 
\mathcal{L}^{\mbox{\tiny \rm odd}}_{\mbox{\tiny \rm odd}}=\frac 1 2 \left( \cN_o(xq)+\tilde
  \cN_o(xq)\right)
- x \left(  \cS_o(x)-\cS_e(xq)\right) \cBa_{\mbox{\tiny \rm odd}}
$$
where $\cBa_{\mbox{\tiny \rm odd}}$ is given by~\eqref{CBaodd} or~\eqref{CBaodd-new}.
\end{Proposition}
\begin{proof}
Recall that the series $\hat{\mathcal{B}}^{\rm(a)}(u)$ given by~\eqref{Bachapeauenu} counts
 alternating heaps of type $B$, and records, with the variable $u$, the
number of points in the first column. Hence $\mathcal{B}^{\rm(a)}_{\mbox{\tiny \rm odd}}$
is simply the odd part (in $u$) of $\hat{\mathcal{B}}^{\rm(a)}(u)$, specialized at $u=1$:
$$
\mathcal{B}^{\rm(a)}_{\mbox{\tiny \rm odd}}=\frac{1}{2}\left(\hat{\mathcal{B}}^{\rm(a)}(1)-\hat{\mathcal{B}}^{\rm(a)}(-1)\right).
$$
Upon using~\eqref{Bachapeauenu}, and noticing that $x$ and $u$ have the same parity in $\cS(x,u)$, 
 we obtain the first result of the proposition.

For the second one, use in the same spirit~\eqref{altSegmentversion2enu} and 
$$\mathcal{L}_{\mbox{\tiny \rm odd}}=\frac{1}{2}\left(\mathcal{L}(1)-\mathcal{L}(-1)\right).$$

Finally, in order to compute  our last series $\mathcal{L}^{\mbox{\tiny \rm odd}}_{\mbox{\tiny \rm odd}}$,
 we will determine the
generating function $\mathcal{L}^{\mbox{\tiny \rm odd}}(u)$ of self-dual alternating
heaps over a path having an odd number of points in the \emm last,
column, where $u$ records the number of points in the first column.
In order to do this, we restrict to heaps with
an odd number of points in the last column the argument that led to~\eqref{cL-eq}. The first three cases
contribute $x\mathcal{B}_{\mbox{\tiny \rm odd}}^{\rm(a)}$, $xuq/(1-u^2q^2)$ and $x^2uq
\mathcal{B}_{\mbox{\tiny \rm odd}}^{\rm(a)}$ respectively.  The fourth one contributes
$$
x\left(\frac 1 {uq}+ uq \right)\left(\mathcal{L}^{\mbox{\tiny \rm odd}}(uq)
  -x\mathcal{B}_{\mbox{\tiny \rm odd}}^{\rm(a)}\right) -x (1-x)\mathcal{B}_{\mbox{\tiny \rm odd}}^{\rm(a)}.
$$
Putting together the four contributions gives:
  $$
\mathcal{L}^{\mbox{\tiny \rm odd}}(u)=
\frac{xuq}{1-u^2q^2}
-\frac{x^2}{uq}(1-uq)\mathcal{B}_{\mbox{\tiny \rm
    odd}}^{\rm(a)}+\frac{x}{uq}(1+u^2q^2)\mathcal{L}^{\mbox{\tiny \rm
    odd}}(uq).
$$
{Iterating this gives the ``odd'' analogue of~\eqref{altSegmentversion2enu}:}
$$
{\mathcal{L}^{\mbox{\tiny \rm odd}}(u)=
\sum_{k\ge 1} (xq)^ku^{2-k} q^{-{k\choose 2}} \frac{
  \llp -u^2q^2\rrp_{k-1}}{1-u^2q^{2k}}
-x\left(\cS(x,u)-u\cS(xq,u)\right) \mathcal{B}_{\mbox{\tiny \rm
    odd}}^{\rm(a)}.}
$$
The expression of $\mathcal{L}^{\mbox{\tiny \rm odd}}_{\mbox{\tiny \rm odd}}$ follows by writing
$$
\mathcal{L}^{\mbox{\tiny \rm odd}}_{\mbox{\tiny \rm
    odd}}=\frac{1}{2}\left(\mathcal{L}^{\mbox{\tiny \rm
      odd}}(1)-\mathcal{L}^{\mbox{\tiny \rm odd}}(-1)\right).
$$
\end{proof}

{\noindent{\bf Remark.} {Theorem}~\ref{thm:newinvolA} gives $\cA$ as  a
  series in $x$ whose coefficients are Laurent polynomials in $q$
  (while the expression of {Theorem}~\ref{thm:ABD-inv} defines a series in $x$
  whose coefficients are rational functions of $q$, with poles at
  roots of unity). In particular, we can specialize {Theorem}~\ref{thm:newinvolA}
  at $q=1$ to recover the number of  fc involutions in
   $\cA_n$, already derived by Stembridge~\cite{St3}. Similarly, the
  expression~\eqref{cBa-new} of $\cBa$ can be specialized at $q=1$. Combined
  with~\eqref{cBna-0} and~\eqref{Dinv-sol}, this allows to recover the
  number of fc involutions in $B_n$ and $D_n$. Finally, note that the
  expression of $\cBa_{\rm odd}$ in Proposition~\ref{prop:newinvolBetsegmentOdd} is also of the
  same type: a series in $q$ whose coefficients are Laurent
  polynomials in $q$. The specialization $q=1$ is again possible. For
  the series $\cL$ and its variants, we obtain the following
  periodicity results.
\begin{Corollary}\label{cor:L-odd}
  Write $\cL(x)=\sum_{n\ge 1} \cL_n x^n$, where $\cL_n\equiv \cL_n(q)$  counts
  self-dual alternating heaps over the $n$-point path. Define similarly the
  series $\cL_{{\rm odd},n}$ (resp. $\cL_{{\rm odd},n}^{\rm odd}$) for
  self-dual alternating heaps with an odd first (resp. first and last)
  column. Then
  $\cL_n$ (resp. $\cL_{{\rm odd},n}$)
is a rational  fraction in $q$
  of denominator $(1-q^n)$ (resp. $(1-q^{2n})$). Consequently, the sequence of its
  coefficients is ultimately periodic, with period dividing $n$
  (resp. $2n$). The series $\cL_{{\rm odd},n}^{\rm    odd}$ is a
  polynomial when $n$ is even, and otherwise a rational  fraction 
  of denominator $(1-q^{2n})$. In this case its coefficients form a
  periodic sequence, with period dividing $2n$.
\end{Corollary}
\begin{proof}
 This follows from  the expressions of $\cL(x)$, $\cL_{\rm odd}(x)$
 and $\cL_{\rm odd}^{\rm odd}(x)$ given in 
  \mbmChangenew{Theorem}~\ref{prop:newinvolBetsegment} and Proposition~\ref{prop:newinvolBetsegmentOdd}.
\end{proof}
}

\section{Generating functions  for infinite types}
\label{sec:infinite}
In this section, we complete our results for the infinite types. The
$\tA$-case has been solved in Section~\ref{sec:At}, and we now derive the  \gfs\ of fully commutative elements,
and  fully commutative involutions, in Coxeter groups of type
$\tB$, $\tC$ and $\tD$.

\subsection{All fully commutative elements}
Recall that all series for finite types have a common form
({Theorem}~\ref{thm:ABD}). This is also true for the infinite types
$\tB, \tC$ and $\tD$. In particular, all involve the \gf\
{$\AL(x)$} of
alternating heaps over a path, determined in the previous section
{(\mbmChangenew{Theorem}~\ref{thm:BL-neg})}.

\begin{Theorem}\label{thm:BCDt-alt}
  Let $\tB\equiv \tB(x,q)$, $\tC\equiv \tC(x,q)$, and $\tD\equiv
  \tD(x,q)$  be the \gfs\ of fully commutative elements of type $\tB$,
  $\tC$ and $\tD$,
  defined respectively by 
$$
\tB(x,q)=\sum_{n\ge 0} \tB_{n}^{FC} (q) x^n,\qquad 
{\tC(x,q)=\sum_{n\ge 0} \tC_{n-1}^{FC} (q) x^{n},}
\qquad
\tD(x,q)=\sum_{n\ge 0} \tD_{n+1}^{FC} (q) x^n.
$$
Let $\AL(x)$ be the \gf\ of alternating heaps over a path, expressed
in \mbmChangenew{Theorem}~\ref{thm:BL-neg}, and define the series $\JJ(x)$ and
$\KK(x)$ by~\eqref{JK-def} as before.
Then each of the series $\tB, \tC$ and $\tD$ can be written as
\beq\label{tilde-form}
R_0 L(x) + R_1 \frac{\KK(x)}{\JJ(x)} + R_2\frac{\JJ(xq)}{\JJ(x)}+R_3 +S,
\eeq
where  the series $R_i$ are
explicit rational series in $x$ and $q$ and $S$ is a simple  $q$-series. More precisely, the series
 $R_0$, $R_1$, $R_2$ and $S$ are given by the following table: 
$$
\begin{array}{l|ccc|}
& B & C & D\\\hline &&&\\
R_0 & 2  & 1  &  4 
\\
R_1 & \displaystyle -\frac{x+q(x-1)+xq^2(x-2)}{1-xq^2}  & 
 \displaystyle 2\ \frac{xq^2(1-x)}{1-xq^2}   
& \displaystyle  4\ \frac{q-x-xq+x^2q^2}{1-xq^2} \\
 R_2 &\displaystyle  \frac{xq^2(1-x)\left( q-x(1+q)+x^2q^2\right)}{(1-xq)(1-xq^2)^2}   & 
\displaystyle   \frac{x^2q^4(1-x)^2}{(1-xq)(1-xq^2)^2} &
\displaystyle   \frac{(q-x-xq+x^2q^2)^2}{(1-xq)(1-xq^2)^2} \\
S&\displaystyle \sum_{n\ge 3} \frac{x^n q^{2(2n-1)}}{1-q^{2n-1}} 
 &  0  &\displaystyle  2 \sum_{n\ge 3} \frac{x^n q^{3n}}{1-q^{n}} 
\end{array}
$$
and, with obvious notation,
\begin{align*}
R_3^{B}&
={\frac {{x}^{3}{q}^{4} \left(- 3\,{q}^{6}x-2\,x{q}^{5}+{q}^{4}x+4\,{q}^
{4}+2\,{q}^{3}-x{q}^{2}-{q}^{2}+q+1 \right) }{ \left( 1-q \right) \left(1- x{q}^{2}
 \right) ^{2} }}
,\\
 R_3^{C}&=
\frac{x^3 q^4(1+q)(1-3q+4q^2+2xq^3-3xq^4)}{(1-q)(1-xq^2)^2},\\
 R_3^{D}&= 
{\frac {{x}^{3}{q}^{5} \left(- 3\,{q}^{6}x-3\,{q}^{5}x+4\,{q}^{4}+3\,{q
}^{3}-2\,x{q}^{2}+q+2 \right) }{ \left( 1-q \right) \left(1- x{q}^{2} \right) ^{2}
 }}
.
\end{align*}
\end{Theorem}
\noindent
{\bf Remarks} 
\\
1. The coefficients of $x^0, x^1$ and $x^2$ {in our three series
 are irrelevant, since the corresponding groups are
not well-defined. We have simply adjusted these coefficients  so as to make the rational part $R_3$ a multiple of
$x^3$.}

\smallskip

\noindent {2. Again, we can feed a computer algebra system
  with our expressions and expand then in $x$ to obtain the series
  $\tilde W_n^{FC}(q)$ for small values of $n$. The first few are reported
  in Table~\ref{tab:affine}.}

\begin{table}[htb]
  \centering
  \begin{tabular}{c|cc|}
 $n$ & 3 & 4 \\
\hline&& \\
$\tA_{n-1}^{FC}$    & $\frac{[1,2,3]}{1-q}$ & $\frac{[1, 4, 9, 12,
                                         8]}{1-q^2}$
\\
$\tB_{n}^{FC}$ & $\frac{(1+q)^2[1, 3, 7, 11, 14, 15, 12, 7, 2, 0, \bar 3, \bar 1,
            \bar 1]}{(1+q+q^2)(1-q^5)}$ & 
$\frac{(1+q)[1, 4, 10, 18, 24, 29, 29, 27, 19, 12, 3, \bar 3, \bar 9, \bar 8, \bar 7, \bar 3, \bar 1, \bar 1]}{1-q^7}$\\
$\tC_{n-1}^{FC}$ & $\frac{[1, 3, 5, 7, 6, 4, 2]}{1-q^3}$
&$\frac{[1, 3, 5, 7, 6, 4, 2, \bar 2, \bar 2, \bar 2]}{1-q}$
\\
$\tD_{n+1}^{FC}$ & $\frac{(1+q)[1,4,10,17,17,13,4,\bar 9,\bar 5,\bar 4]}{1-q^3}$ 
& $\frac{[1, 6, 20, 46, 78, 100, 100, 78, 35, \bar 14, \bar 36, \bar 44, \bar 36, \bar 16, \bar 8, \bar 4]}{1-q^4}$ 
  \end{tabular}
\vskip 3mm
  \caption{Length generating functions of fc elements in affine Coxeter
    groups. The list $[a_0, \ldots, a_k]$ stands for the polynomial
    $a_0+a_1q+\cdots + a_k q^k$, and the notation $\bar a$ means
    $-a$. The $\tA$-results come from {Theorem}~\ref{thm:O}.}
  \label{tab:affine}
\end{table}

{Before we prove {Theorem}~\ref{thm:BCDt-alt}, let us
  show that it implies the periodicity results
  of~\cite[Thm.~4.1]{BJN-long}. \mbmChangenew{Note that the
    \emph{exact} values of the periods  were found in~\cite{JN-periods}.}
  \begin{Corollary}
    The  coefficients of the series $\tB_n^{FC}(q)$
    (resp. $\tC_{n-1}^{FC}(q)$,     $\tD_{n+1}^{FC}(q)$) form a
    periodic sequence of period dividing $n(2n-1)$ (resp. $n$, $n$).
  \end{Corollary}
  \begin{proof}
    We start from the expressions of {Theorem}~\ref{thm:BCDt-alt}. The series
    $R_1$ and $R_2$, once expanded in $x$, have polynomial
    coefficients in $q$. The same holds for $K(x)/J(x)$ (by~\eqref{Ba-expr})
    and $J(xq)/J(x)$ (by~\eqref{A-expr}). The coefficient of $x^n$ in $R_3$
    is a fraction with denominator $(1-q)$. Then the results follow
    from the properties of $L_n(q)$ (Corollary~\ref{cor:L}) and the values
    of the series $S$.
  \end{proof}
}

\begin{proof}[Proof of {Theorem}~\ref{thm:BCDt-alt}]
 We begin with the simplest case, which is $\tC_n$. Recall that the
Coxeter graph of $\tC_n$ is the $(n+1)$-point path (Figure~\ref{fig:dynkin}), which explains why the \gf\ we consider is $\sum_{n}
  \tC_n^{FC} x^{n+1}= {\sum_{n}
  \tC_{n-1}^{FC} x^{n}}$. 

   For $n\ge 2$, {Theorem~3.4}
in~\cite{BJN-long} describes the set of heaps
  encoding fc elements of $\tC_n$ as the disjoint union of  five sets, which we
  will enumerate separately. 
  \begin{itemize}
  \item First come alternating heaps over the $(n+1)$-point path. The
    associated \gf\ is 
$$
L(x) -x L_1 -x^2L_2,
$$
where $L_i$ counts alternating heaps over the $i$-point
path. {Since the coefficients of $x$ and $x^2$ in our
  final series are irrelevant, we do not need the expressions of $L_1$
  or $L_2$.}
\item Then come \emm zigzags,. Their \gf , derived
  in the proof of~\cite[Prop.~4.2]{BJN-long}, is
$$
\sum_{n\ge 2} x^{n+1}\left( \frac{2n q^{2n+2}}{1-q} + (2n-2) q^{2n+1}
\right)=
2\,{\frac {{x}^{3}{q}^{5} \left(1+q -x{q}^{3} \right) }{ \left(1- q
 \right)  \left(1- x{q}^{2} \right) ^{2}}}
.
$$

\item Then come \emm left peaks,. They are obtained as follows:
  {starting from an alternating heap over the $(n-j+1)$-point path
  (with $1\le j \le n-1$, and vertices indexed by $s_j, \ldots,
  s_{n-1},u$),  having  exactly one point in
  the $s_j$-column, one inflates this point into a 
$<$-shaped heap}
  $s_j s_{j-1} \cdots s_1 t s_1 \cdots s_{j-1} s_j$ (this is closely related to the description of
  non-alternating fc elements of type $B$ in Section~\ref{sec:recB}). The
  corresponding \gf\ is
\beq\label{LP-B}
LP(x)=\frac{xq^2}{1-xq^2} \left( \Ba(x) -1-xq-x \Ba(x)\right),
\eeq
where the term between parentheses counts alternating heaps with at
least two columns and exactly one point in the first column.
\item The \gf\ $RP(x)$ of \emm right peaks, coincides with $LP(x)$.
\item The fifth class consists of \emm left-right peaks,. To  obtain
  them, one starts from  an alternating heap $H$ over a path with
  vertices  $s_j,
  \ldots,s_k$, having one point in its first and last columns
  (with $1\le j<k\le n-1$), {inflates the point in the
    $s_j$-column into a $<$-shaped heap
  $s_j s_{j-1} \cdots s_1 t s_1 \cdots s_{j-1} s_j$, and symmetrically
  inflates the point in the $s_k$-column  into a $>$-shaped heap $s_k s_{k+1} \cdots s_{n-1} u  s_{n-1}\cdots s_{k+1}s_k
  $}. Accordingly, the \gf\ is
\beq\label{LRP-A}
{LRP(x)}= \left(\frac{xq^2}{1-xq^2}\right)^2 \left(A(x)-1-2xA(x)+x^2A(x)+x-xq\right),
\eeq
where the term between parentheses counts alternating heaps of type
$A$ with at least two columns and one point in the first and
last columns. 
  \end{itemize}
 {We now add all contributions, and inject the expressions~\eqref{A-expr}
 and~\eqref{Ba-expr} of $A$ and $\Ba$. This gives the
form~\eqref{tilde-form} for the series
$$\sum_{n\ge 2} \tC_n^{FC} x^{n+1} = \sum_{n\ge 3} \tC_{n-1}x^n,$$
 for the values of $S, R_0,
R_1, R_2$  listed above, but with a different value of $R_3$. Removing
from $R_3$ its terms in $x^1$ and $x^2$ gives the announced value of $R_3^C$.}

\medskip
We now move to the $\tB$-case. For $n\ge 2$, Section 3.2
in~\cite{BJN-long} describes the heaps encoding fc elements of
$\tB_{n+1}$ in terms of   a transformation $\Delta_t$ acting on fc heaps
of type $\tC_n$ 
(which we have described above).  This transformation maps an fc heap
$H$ of
type $\tC_n$ on a \emm set, of heaps of type $\tB_{n+1}$, obtained by
replacing each occurrence of the generator $t$ in $H$ by one or
several copies of $t_1$ and/or $t_2$ (with the notation of Figure~\ref{fig:dynkin}). The union of
these sets is disjoint, and consists of all fc heaps of type
$\tB_{n+1}$. Therefore we now count heaps of $\Delta_t(X)$, for each of
the five families~$X$ of fc heaps of type $\tC_n$ listed above.
\begin{itemize}
\item Given an alternating heap $H$ of $\tC_n$, the length \gf\ of the set
  $\Delta_{t}(H)$ is $q^{|H|}\alpha_t(H) $, where
 \beq\label{alpha}
    \alpha_t(H)=2+
  \begin{cases}
0 & \hbox{if } H \hbox{ contains at least two occurrences of } t,
\\
-1 &\hbox{if } H \hbox{ contains no occurrence of } t,
\\
q &\hbox{if } H \hbox{ contains exactly one occurrence of } t.
 \end{cases}
\eeq
 In the second case, erasing the first (empty) column of $H$ leaves
    an alternating heap of type~$B$ on $n\ge 2$ points, reflected in a
    vertical line.
In the third case, $H$
  is  an alternating heap of type $B$  on $n+1\ge 3$ points with a non-empty
  last column, reflected in a vertical line.
Hence, the \gf\ of fc heaps of type $\tB_{n+1}$ arising from
alternating heaps of type $\tC_n$ is
\begin{multline}\label{Dt-alt}
  2\left( L(x)-xL_1-x^2L_2\right) -x \left( \Ba(x) -1 -x \Ba_1\right)
  \\
+q
\left( \Ba(x)-1-x\Ba_1 -x^2\Ba_2 -x\Ba(x)+x +{x^2} \Ba_1\right),
\end{multline}
where $\Ba_i$ counts alternating heaps of type $B$ over the $i$-point
path.
{Again, we do not need the expressions of $L_1$, $L_2$, $\Ba_1$ or $\Ba_2$.}
\item The \gf\ of heaps arising from zigzags is given
  in~\cite[Prop.~4.3]{BJN-long}. It reads
  \begin{multline*}
    \sum_{n\ge 2} x^{n+1} \left( \frac{(2n+3)q^{2n+4}}{1-q} +
  \frac{q^{2(2n+1)}}{1-q^{2n+1}} +(2n+2) q^{2n+3} +(2n-2) q^{2n+2}\right)
\\=
\sum_{n\ge 2} x^{n+1} \frac{q^{2(2n+1)}}{1-q^{2n+1}} 
+{\frac {{x}^{3}{q}^{6} \left(- {q}^{4}x-4\,{q}^{3}x+{q}^{2}+4\,q+2
 \right) }{ \left(1- q \right)  \left(1- x{q}^{2} \right) ^{2}}}
.
  \end{multline*}
\item Each left peak $H$ of $\tC_n$ gives rise to one element of
  $\tB_{n+1}$, of size $|H|+1$. The associated \gf\ is thus
$
q LP(x)$, with $LP(x)$ given by~\eqref{LP-B}.
\item Right peaks behave under the map $\Delta_t$ as alternating
  elements do, with one special case: the right peak $H$
  corresponding to the element $ts_1 \cdots
  s_{n-1}u s_{n-1}\cdots s_1 t$ does not give rise to two
  {heaps of size $|H|$}, but to nine heaps of
  $\tB_{n+1}$: four of size $|H|$, four of size $|H|+1$ and one of
  size $|H|+2$. By adapting the argument that led us
  to~\eqref{Dt-alt}, we obtain the \gf\ of heaps arising from right peaks as
  \begin{multline}\label{RpDt}
  RP_{\Delta_t}=  2 RP(x)- \frac{x^2q^2}{1-xq^2} \left(A(x)-1-xA(x)\right)+\left(2 + 4q+q^2\right)\sum_{n\ge 2} x^{n+1} q^{2n+1}
\\+\frac{xq^3}{1-xq^2}  \left( A(x)-1-x(1+q)  -2x(A(x)-1)+x^2A(x)\right),
  \end{multline}
where $RP(x)=LP(x)$ is given by~\eqref{LP-B}. 
{In the above expression, the second term counts right peaks
  with an empty first column, and the fourth one is $q$ times the \gf\
  of right peaks having exactly one point in the first column.}
\item Each left-right peak $H$ of $\tC_n$ gives rise to one element of
  $\tB_{n+1}$, of size $|H|+1$. The associated \gf\ is thus
$
q LRP(x)$, with $LRP(x)$ given by~\eqref{LRP-A}.
\end{itemize}
Adding all contributions, {and injecting the expressions~\eqref{A-expr}
 and~\eqref{Ba-expr} of $A$ and $\Ba$,}  gives the form~\eqref{tilde-form} for
the series
$$
\sum_{n\ge 2} \tB_{n+1}^{FC}(q)x^{n+1}= \sum_{n\ge 3}
\tB_{n}^{FC}(q)x^{n} , 
$$
for the values of  $S, R_0, R_1$ and $ R_2$ listed above. {The value of
$R_3$ that we obtain only differs from the one given in the
{theorem} by the coefficients of $x^1$ and $x^2$.}

\medskip
Let us finally address  the $\tD$-case. Again, Section 3.2
in~\cite{BJN-long} describes the heaps encoding fc elements of
$\tD_{n+2}$ (for $n\ge 2$) in terms of   a transformation $\Delta_{t,u}$ acting on fc heaps
of type $\tC_n$. The transformation $\Delta_t$ {that we used to
generate elements of $\tB_{n+1}$} was acting on the first
column of a heap. {We can roughly describe the new
  transformation  $\Delta_{t,u}$ by saying that it also modifies
the last column, in a symmetric fashion. We refer
to~\cite{BJN-long} for details.}   
We will now count heaps of $\Delta_{t,u}(X)$, for each of
the five families $X$ of fc heaps of type $\tC_n$ listed above.
\begin{itemize}
\item Given an alternating heap $H$ of $\tC_n$, the length \gf\ of the set
  $\Delta_{t,u}(H)$ is $q^{|H|}\alpha_t(H) \alpha_u(H)$, where 
$
    \alpha_t(H)$ is defined by~\eqref{alpha}, and $\alpha_u(H)$ is
    defined similarly in terms of the generator $u$. A careful case by
    case study gives the \gf\ of fc elements of type {$\tD_{n+2}$}
    arising from alternating heaps of $\tC_n$ as
    \begin{multline*}
      4 L(x) -3 x^2A(x)+ 2 (q-2) x \big(A(x)-xA(x)\big) -4 x \big(\Ba(x)-A(x)\big)\\+
(4q+q^2)\big(A(x)-2{x}A(x)+x^2A(x)\big) +4q \big(\Ba(x)-x\Ba(x)-A(x)+xA(x)\big) +
\Pol_2(x),
 \end{multline*}
where $\Pol_2(x)$ is a polynomial in $x$ of degree at most 2 (recall
that the first three coefficients of our final series $\tD(x,q)$ are
irrelevant). The first term corresponds to the generic case where a
heap gives rise to 4 heaps of the same size. The
second (resp. third, fourth, fifth, sixth) term corrects this contribution for
heaps $H$ such that $\{|H|_t,|H|_u\}=\{0\}$ (resp. $\{0,1\}$,
$\{0,{i}\}$,   $\{1\}$, $\{1,{i}\}$,
{with $i\ge 2$}).
\item The \gf\ of heaps arising from zigzags is given
  in~\cite[Prop.~4.4]{BJN-long}. It reads
  \begin{multline*}
    \sum_{n\ge 2} x^{n+1} \left( \frac{(2n+6)q^{2n+5}}{1-q} +
  \frac{2q^{3(n+1)}}{1-q^{n+1}} +(2n+4) q^{2n+4} +(2n-2) q^{2n+3}\right)
\\=
2\sum_{n\ge 3} x^{n} \frac{q^{3n}}{1-q^n} 
+ 2\,{\frac {{x}^{3}{q}^{7} \left(- {q}^{4}x-3\,{q}^{3}x+{q}^{2}+3\,q+1
 \right) }{ \left(1- x{q}^{2} \right) ^{2} \left( 1-q \right) }}
.
  \end{multline*}
\item Applying $\Delta_{t,u}$ to a right peak boils down to applying
  $\Delta_t$, and then replacing the only occurrence of $u$ by $u_1
  u_2$. Hence the \gf\ of heaps of $\tD_{n+2}$ arising from a right
  peak of $\tC_n$ is $q RP_{\Delta_t}$, where  $ RP_{\Delta_t}$ is
  given by~\eqref{RpDt}.
\item The case of left peaks is symmetric, and gives another term $q
  RP_{\Delta_t}$.
\item Finally, each left-right peak $H$ of $\tC_n$ gives rise to one
  element of {$\tD_{n+2}$}, of size $|H|+2$. The associated \gf\ is thus
  $q^2LRP(x)$, with $LRP(x)$ given by~\eqref{LRP-A}.
\end{itemize}
Adding all contributions gives an expression of the form~\eqref{tilde-form} for
the series
$$
\sum_{n\ge 2} \tD_{n+2}^{FC}(q)x^{n+1}= \sum_{n\ge 3} \tD_{n+1}^{FC}(q)x^{n},
$$
for the values of $S, R_1, R_2$ listed in the {theorem}. The series
$R_3$ only differs from $R_3^D$ by a
polynomial in $x$ of degree 2.
\end{proof}

\subsection{Fully commutative involutions}

We now establish the corresponding results for fc involutions.
Again, the three series that we obtain have a similar form. 
In particular, each of them involves the  generating function
$\cL(x)$ of
self-dual alternating heaps over a path, determined in the previous
section (\mbmChangenew{Theorem}~\ref{prop:newinvolBetsegment}), or one of its variants obtained by
imposing parity conditions on the first and/or last column
(Proposition~\ref{prop:newinvolBetsegmentOdd}). Recall that the
$\tA$-case was solved in Section~\ref{sec:At}.

\begin{Theorem}\label{thm:invol-BCDt-alt}
  Let $\tilde{\cB}\equiv \tilde{\cB}(x,q)$, $\tilde{\cC}\equiv \tilde{\cC}(x,q)$, and $\tilde{\cD}\equiv
  \tilde{\cD}(x,q)$  be the \gfs\ of fully commutative involutions of type $\tB$,
  $\tC$ and $\tD$,
  defined respectively by 
$$
\tilde{\cB}(x,q)=\sum_{n\ge 0} \tilde{\cB}_{n}^{FC} (q) x^n,\qquad 
\tilde{\cC}(x,q)=\sum_{n\ge 0} \tilde{\cC}_{n-1}^{FC} (q) x^{n},
\qquad
\tilde{\cD}(x,q)=\sum_{n\ge 0} \tilde{\cD}_{n+1}^{FC} (q) x^n.
$$
 Define the series $\JI(x)$ and 
$\KI(x)$ by~\eqref{DD-def} and~\eqref{UV-def}  as before, and let
$\KI_e(x)$ denote the even part of $\KI(x)$ in $x$ (see~\eqref{eop}).
Then each of the series $\tilde{\cB}, \tilde{\cC}$ and $\tilde{\cD}$
can be written as 
\beq\label{inv-tilde-form}
\cR_0 \cL_0(x) +\cR\frac{\KI_e(xq)}{\JI(x)}+ \cR_1 \frac{\KI(x)}{\JI(x)} + \cR_2\frac{\JI(-xq)}{\JI(x)}+\cR_3 +\cS,
\eeq
where $\cL_0(x)$ is the \gf\ of
self-dual alternating heaps
(\mbmChangenew{Theorem}~\ref{prop:newinvolBetsegment}) or one of its variants
(Proposition~\ref{prop:newinvolBetsegmentOdd}),  $\cR$ and the $\cR_i$ are
explicit rational series in $x$ and $q$, and  $\cS$ is a simple $q$-series. The series
 $\cR_0$, $\cL_0$,  $\cR$, $\cR_1$, $\cR_2$  and $\cS$ are given by the following table: 
$$\begin{array}{l|ccc|}
& \cB & \cC & \cD \\
\hline&&& \\
\cR_0& 2  & 1  & 4 \\
    \cL_0 &\cL_{\rm odd} &  \cL &
\cL_{\rm odd}^{\rm odd} \\
\cR &\displaystyle 2x^2q^3\ \frac{1-x}{1-xq^2}  &  0  &
       \displaystyle  4xq\ \frac{x+q(1-x)-x^2q^2}{1-xq^2}  \\
\cR_1 & \displaystyle \frac{x+q(1-x)-x^2q^2}{1-xq^2} &
\displaystyle{2\ \frac{xq^2(1-x)}{1-xq^2}   }
& 0
\\
 \cR_2& \displaystyle \frac{xq^2(1-x)\left( x+q(1-x)-x^2q^2\right)}{(1-xq^2)^2} &
\displaystyle  \frac{x^2q^4(1-x)^2}{(1-xq^2)^2}&\displaystyle 
\frac{(x+q(1-x)-x^2q^2)^2}
{(1-xq^2)^2}\\
\cS &\displaystyle \sum_{n\geq3}\frac{x^nq^{2(2n-1)}}{1-q^{2n-1}} 
 &  0  &\displaystyle  2\sum_{n\geq3}\frac{x^nq^{4n}}{1-q^{2n}}
\end{array}$$
and, with obvious notation,
\begin{align*}
\cR_3^{B}&
=x^3q^4\,{\frac {x(q^2-2q^3+3q^4-3q^6)-1+3q-5q^2+4q^4}{ \left( 1-q \right) \left(1- x{q}^{2}
 \right) ^{2} }}
,\\
 \cR_3^{C}&={\frac{x^3 q^4(
1-3q-5q^2+(2x+5)q^3+(3x+4)q^4-4xq^5-3xq^6)}{(1-q^2)(1-xq^2)^2},}
\\
 \cR_3^{D}&= x^3q^5\,
{\frac {xq^2(2-2q+3q^3-2q^4-3q^5)-2+3q-q^2-5q^3+3q^4+4q^5 }{ \left( 1-q^2 \right) \left(1- x{q}^{2} \right) ^{2}
 }}
.
\end{align*}
\end{Theorem}
\noindent
{\bf Remark.} As before, the coefficients of $x^0, x^1$ and $x^2$ in our series
 are irrelevant, {and we have adjusted them so that the
   fractions $\cR_3$ are multiples of $x^3$. Expanding our series in~$x$ gives the first few
 values of the series $\tilde {\mathcal W}_n^{FC}(q)$, reported in Table~\ref{tab:affine-inv}.}

\begin{table}[htb]
  \centering
  \begin{tabular}{c|cc|}
 $n$ & 3 & 4 \\
\hline&& \\
$\ctA_{n-1}^{FC}$    & [1,3]& $\frac{[1, 4, 2, 0, 3, \bar 4]}{1- q^4}$ 
\\
$\ctB_{n}^{FC}$ &$ \frac{[1, 5, 8, 11, 14, 16, 16, 11, 8, 3, 1, \bar 5, \bar 3, \bar 4,
             \bar 2, \bar 2, \bar 1, \bar 1]}{(1-q^5)(1-q^6)/(1-q)}$
& $ \frac{[1,6,11,10,12,13,13,12,6,\bar 1,\bar 2,\bar 5,\bar 6,\bar 7,\bar 6,\bar 5,\bar 4,\bar 2,\bar 1,\bar 1]}{(1+q)(1-q^7)}$\\
$\ctC_{n-1}^{FC}$ &$\frac{[1, 4, 4, 4, 1, 2, 2]}{(1+q)(1-q^3)}$&
 $\frac{[1, 4, 2, 0, 3, \bar 2, 0, 4, 0, \bar 4, \bar 2]}{1-q^2}$
 \\
$\ctD_{n+1}^{FC}$& $\frac{[1, 5, 6, 4, 7, 4, 12, \bar 1, 0, \bar 4, \bar 1, \bar 4, \bar 1,
              \bar 4]}{1-q^6}$
&$\frac{[1, 6, 10, 6, 7, 6, 8, 4, 5, \bar 2, \bar 8, \bar 6, \bar 5, \bar 6, \bar 6, \bar 4, \bar 2, \bar 4]}{1-q^8}$
  \end{tabular}
\vskip 3mm
  \caption{Length generating functions of fc involutions in affine Coxeter
    groups. The list $[a_0, \ldots, a_k]$ stands for the polynomial
    $a_0+a_1q+\cdots + a_k q^k$, and the notation $\bar a$ means
    $-a$. The $\tA$-results come from {Theorem}~\ref{thm:At-inv}. Recall that $\ctA_{2m}(q)$ is a polynomial.}
  \label{tab:affine-inv}
\end{table}

{Before we prove {Theorem}~\ref{thm:invol-BCDt-alt}, let us
  show that it implies the periodicity results
  of~\cite[Prop.~3.4]{BJN-inv}. \mbmChangenew{Again, we refer
    to~\cite{JN-periods} for the \emph{exact} values of the periods.}
  \begin{Corollary}
    The  coefficients of the series $\ctB_n^{FC}(q)$
    (resp. $\ctC_{n-1}^{FC}(q)$,     $\ctD_{n+1}^{FC}(q)$) form a
    periodic sequence of period dividing $2n(2n-1)$ (resp. $2n$,
    $2n$). Moreover, if $n$ is even, the period of
    $\ctC_{n-1}^{FC}(q)$ divides $n$.
  \end{Corollary}
  \begin{proof}
    We start from the expressions of {Theorem}~\ref{thm:invol-BCDt-alt}. The series
    $\cR$, $\cR_1$ and $\cR_2$, once expanded in $x$, have polynomial
    coefficients in $q$. The same holds for $\KI_e(xq)/\JI(x)$
    (by~\eqref{CBaodd}),  $\KI(x)/\JI(x)$ (by~\eqref{cBalt})
    and $\JI(-xq)/\JI(x)$ (by~\eqref{cA-expr}). The coefficient of $x^n$ in $\cR_3$
    is a fraction with denominator $(1-q)$ (in the $B$-case) or
    $(1-q^2)$ (in the $C$- and $D$-cases). Then the results follow
    from the properties of $\cL_n(q)$ and its variants (Corollary~\ref{cor:L-odd}) and the values
    of the series $\mathcal{S}$.
  \end{proof}
}

\begin{proof}[Proof of Theorem~\ref{thm:invol-BCDt-alt}]
By~\cite[Lemma~1.4]{BJN-inv}, a heap encodes an involution if and
only if it is self-dual. Hence we revisit  the proof of
Theorem~\ref{thm:BCDt-alt}, by restricting the enumeration to
self-dual heaps.

Again we begin with the simplest case, which is $\tC_n$, and we
enumerate  self-dual heaps in the five disjoint sets
listed in~\cite[Theorem~3.4]{BJN-long}. The expressions that we obtain are perfect
analogues of those determined in the proof of Theorem~\ref{thm:BCDt-alt}.
  \begin{itemize}
  \item For self-dual alternating heaps over the $(n+1)$-point path, the \gf\ is 
$$
\cL(x) -x \cL_1 -x^2\cL_2,
$$
where $\cL_i$ counts self-dual alternating heaps over the $i$-point
path. 
\item The \gf\ for self-dual  zigzags, derived
  in the proof of~\cite[Prop.~3.4]{BJN-inv}, is
$$
\sum_{n\ge 2} \frac{2 q^{2n+3}}{1-q^2}x^{n+1}=
{\frac {2x^3 q^7}{ \left(1- q^2
 \right)  \left(1- x{q}^{2} \right)}}
.
$$
\item For  self-dual left peaks, 
the \gf\ is the analogue of~\eqref{LP-B}:
\beq\label{inv-LP-B}
{\mathcal {LP}}(x)=\frac{xq^2}{1-xq^2} \left( \cBa(x) -1-xq-x \cBa(x)\right).
\eeq
\item The \gf\ ${\mathcal {RP}}(x)$ of  self-dual right peaks coincides with ${\mathcal {LP}}(x)$.
\item Finally, for self-dual left-right peaks we obtain the self-dual
  analogue of~\eqref{LRP-A}:
\beq\label{inv-LRP-A}
{\mathcal {LRP}}(x)=\left(\frac{xq^2}{1-xq^2}\right)^2 \left(\cA(x)-1-2x\cA(x)+x^2\cA(x)+x-xq\right).
\eeq
  \end{itemize}
 Adding all contributions, {and injecting the
   expressions~\eqref{cA-expr} and~\eqref{cBalt} of $\cA$ and $\cBa$}
gives the
form~\eqref{inv-tilde-form} for $\tilde{\cC}(x,q)$, for the values of
$\cR_0$, $\cL_0$, $\cR$, 
$\cR_1$,  $\cR_2$  and $\cS$  listed above. The value of $\cR_3$ is
different, but one obtains $\cR_3^C$ by subtracting from $\cR_3$ its
terms in $x^1$ and $x^2$.

\medskip
We now move to the $\tB$-case.  
In the proof of
Theorem~\ref{thm:BCDt-alt}, we have described the set of fc elements of
$\tB_{n+1}$ as the disjoint union of sets  $\Delta_t(H)$, for $H$ an
fc element of  $\tC_n$. Now we have to determine, for each $H$, which
elements in $\Delta_t(H)$ are self-dual. Fortunately, the answer is
simple: if $H$ itself is not self-dual, then no element in
$\Delta_t(H)$ is self-dual. If $H$ is self-dual, then  all
elements of $\Delta_t(H)$ are self-dual, except in the two following
cases:
\begin{enumerate}
\item[(i)] if $H$ is the special right peak $ts_1 \cdots
  s_{n-1}u s_{n-1}\cdots s_1 t$, then exactly three of the nine elements of
  $\Delta_t(H)$ are self-dual: two of size $|H|$ and one of size
  $|H|+2$;
\item[(ii)] if $H$ is either alternating or a (non-special) right peak,
  \emm and has a positive even number of points in its first
  column,, then  $\Delta_t(H)$ contains no self-dual heap.
\end{enumerate}
In particular, if $H$ is either alternating or a non-special right peak, the
length \gf\ of self-dual heaps in $\Delta_t(H)$ is $q^{|H|}
a_t(H)$ with
\beq\label{a-t-def}
     a_t(H)=
  \begin{cases}
2 &\hbox{if } H \hbox{ contains an odd number of occurrences of } t,
\hbox{ at least equal to 3,}\\
1 &\hbox{if } H \hbox{ contains no occurrence of } t,
\\
2+q &\hbox{if } H \hbox{ contains exactly one occurrence of } t,
\\
0 & \hbox{otherwise.}
  \end{cases}
\eeq

With these changes in mind, we revisit the derivation of the
series $\tB(x,q)$ performed in the proof of Theorem~\ref{thm:BCDt-alt}, and adapt it to self-dual heaps. 
\begin{itemize}
\item  The contribution of involutions of $\tB_{n+1}$ obtained from self-dual alternating heaps of $\tC_n$ is
  found to be:
\begin{multline*}
  2\left( \cL_{\rm odd}(x)-{x \cL_{\rm odd, 1}-x^2 \cL_{\rm odd, 2}}
\right) +x \left( \cB^{\rm(a)}(x) -1 -x \cB^{\rm(a)}_1\right)
  \\
+q
\left( \cB^{\rm(a)}(x)-1-x\cB^{\rm(a)}_1 -x^2\cB^{\rm(a)}_2 -x\cB^{\rm(a)}(x)+x +{x^2} \cB^{\rm(a)}_1\right),
\end{multline*}
where $\cL_{{\rm odd}, i}$ (resp. $\cB^{\rm(a)}_i$) counts self-dual
alternating heaps (resp.~of type $B$) over the $i$-point path.
{The main difference with~\eqref{Dt-alt} is, in the first term,
  the restriction to heaps with an odd number of points in the first
  column. This also results in a sign change in the second term.}
 
\item The \gf\ of self-dual heaps arising from zigzags is given
  in the proof of~\cite[Prop.~3.4]{BJN-inv} and reads
$$
    \sum_{n\ge 2} x^{n+1} \left( \frac{q^{2n+4}}{1-q} +
  \frac{q^{2(2n+1)}}{1-q^{2n+1}}\right)=
\sum_{n\ge 2} x^{n+1} \frac{q^{2(2n+1)}}{1-q^{2n+1}} 
+{\frac {{x}^{3}{q}^{8} }{ \left(1- q \right)  \left(1- x{q}^{2} \right) }}
.
$$
\item Each self-dual 
left peak $H$ of $\tC_n$ gives rise to one fc involution of
  $\tB_{n+1}$, of size $|H|+1$. The associated \gf\ is thus
$
q {\mathcal {LP}}(x)$, with ${\mathcal {LP}}(x)$ given by~\eqref{inv-LP-B}.
\item  For right peaks, we adapt the derivation of~\eqref{RpDt} to the self-dual
  case, keeping in mind the above restrictions (i) and (ii). We obtain the following \gf:
  \begin{multline}\label{inv-RpDt}
  {\mathcal RP}_{\Delta_t}=  2\,
{\mathcal{RP}_{\rm odd}(x)}
+ \frac{x^2q^2}{1-xq^2} \left(\cA(x)-1-x\cA(x)\right)
\\+\frac{xq^3}{1-xq^2}  \left( \cA(x)-1-x(1+q)  -2x(\cA(x)-1)+x^2\cA(x)\right)
+\left(2 + q^2\right)\sum_{n\ge 2} x^{n+1} q^{2n+1},
  \end{multline}
where {$\mathcal{RP}_{\rm odd}(x)$ counts self-dual right
peaks with an odd number of points in the first column (and at least
three columns). The main difference with~\eqref{RpDt} is, as before, the
restriction on the parity of the first column in the first term. This
results in a sign change in the second term. 

Since self-dual right
peaks with an odd first column are obtained by inserting a $>$-shaped
heap to the right of a  self-dual alternating
heap of type $B$ with an odd first column, which gives:}
$$
\mathcal{RP}_{\rm odd}(x)=
\frac{xq^2}{1-xq^2}\left(\cB_{\rm odd}^{\rm(a)}-xq-x\cB_{\rm
    odd}^{\rm(a)}\right),
$$ 
which should be compared to~\eqref{LP-B}.
\item Finally, for self-dual heaps arising from left-right peaks, we
  obtain the series $q\mathcal{LRP}(x)$, where $\mathcal{LRP}(x)$ is
  given by~\eqref{inv-LRP-A}.
\end{itemize}
Adding all contributions,  {and injecting the
   expressions~\eqref{cA-expr},~\eqref{cBalt} and~\eqref{CBaodd} of
   $\cA$, $\cBa$ and $\cBa_{\rm odd}$}, gives an expression of the
form~\eqref{inv-tilde-form} for the series
$$
\sum_{n\ge 2} \tilde{\cB}_{n+1}^{FC}(q)x^{n+1}= \sum_{n\ge 3} \tilde{\cB}_{n}^{FC}(q)x^{n},
$$
for the values of $\cR_0, \cL_0, \cR,  \cR_1$, $\cR_2$ and $\cS$ listed in
the {theorem}. The value of $\cR_3$ only differs from
$\cR_3^{B}$  by the coefficients of $x^1$ and $x^2$.

\medskip
Let us finally address  the $\tD$-case. As in the proof of
Theorem~\ref{thm:BCDt-alt},
we use the transformation $\Delta_{t,u}$ acting 
on fc heaps $H$
of type $\tC_n$. Again, the description of self-dual heaps found in
$\Delta_{t,u}(H)$ is simple. If $H$ is not self-dual, then no heap of
$\Delta_{t,u}(H)$ will be. If $H$ is self-dual, then all heaps of
$\Delta_{t,u}(H)$  are self-dual, with the following exceptions:
\begin{enumerate}
\item [(i)] if $H$ is the special left peak $us_{n-1}\cdots s_1 t s_1
  \cdots s_{n-1} u$ or the special right peak $ t s_1 \cdots s_{n-1} u$
  $s_{n-1} \cdots s_1 t$, then exactly {three} elements of $\Delta_{t,u}(H)$ are
  self-dual: two of size $|H|+1$, one of size $|H|+3$; 
\item [(ii)] if $H$ is either alternating, or a non-special left peak, or a
  non-special right peak \emm with a positive even number of points it
  its first or last column,, then no heap of $\Delta_{t,u}(H)$ is
  self-dual.
\end{enumerate}
In particular, if $H$ is  alternating,
the \gf\ of self-dual elements in
  $\Delta_{t,u}(H)$ is $q^{|H|}a_t(H)a_u(H)$, where $a_t$ is defined
  by~\eqref{a-t-def}, and $a_u$ is defined symmetrically.

   We now  count self-dual heaps of $\Delta_{t,u}(X)$, for each of
the five families $X$ of self-dual fc heaps of type $\tC_n$ listed above.
\begin{itemize}
\item For heaps arising from self-dual alternating heaps of $\tC_n$, a careful case by
    case study involving the parity of the first and last columns gives the \gf\ as:
    \begin{multline*}
 4 \cL_{\rm odd}^{\rm odd}(x) +x^2\cA(x)+2(2+q)x(1-x) \cA (x) +4x \big(
 \cBa_{\rm odd}(x) - (1-x) \cA(x)\big)\\
+(4q+q^2) (1-x)^2 \cA(x) + 4q \big( (1-x) \cBa_{\rm odd}(x) -(1-x)^2
\cA(x) \big)
+\Pol_2(x),
 \end{multline*}
where $\Pol_2(x)$ is a polynomial in $x$ of degree at most 2. {The first term corresponds to the generic case of a heap with odd
first and right columns. The second (resp. third, fourth, fifth,
sixth) term corrects or completes this contribution for heaps $H$ such that
$\{|H|_t, |H|_u\}= \{0\}$ (resp. $\{0,1\}$, $\{0,i\}$, $\{1\}$,
$\{1,i\}$), with $i\ge 3$ odd.}

\item The \gf\ of self-dual heaps arising from zigzags is given
  {at the end of the proof of}~\cite[Prop.~3.4]{BJN-inv} and reads
  $$
    \sum_{n\ge 2} x^{n+1} \left( \frac{2q^{2n+6}}{1-q^2} +
  \frac{2q^{4n+4}}{1-q^{2n+2}}\right)
\\=
2\sum_{n\ge 3} x^{n} \frac{q^{4n}}{1-q^{2n}} 
+ 2\,{\frac {{x}^{3}{q}^{10}  }{ { \left( 1-x{q}^{2} \right) 
\left( 1-q^2 \right)} }}
.
 $$
\item {If $H$ is a self-dual right peak}, the self-dual elements
  of $\Delta_{t,u}(H)$ are obtained from self-dual elements of
  $\Delta_t(H)$ by replacing their only occurrence of $u$ by $u_1
  u_2$. Hence the \gf\ of heaps of $\tD_{n+2}$ arising from a self-dual right
  peak of $\tC_n$ is $q {\mathcal {RP}}_{\Delta_t}$, where  $ {\mathcal {RP}}_{\Delta_t}$ is
  given by~\eqref{inv-RpDt}.
\item The case of self-dual left peaks is symmetric, and gives another term $q
  {\mathcal {RP}}_{\Delta_t}$.
\item Finally, each self-dual left-right peak $H$ of $\tC_n$ gives rise to one
  fc involution of {$\tD_{n+2}$}, of length $|H|+2$. The associated \gf\ is thus
  $q^2{\mathcal {LRP}}(x)$, with ${\mathcal {LRP}}(x)$ given by~\eqref{inv-LRP-A}.
\end{itemize}
Adding all contributions, {and injecting the
   expressions~\eqref{cA-expr} and~\eqref{CBaodd} of
   $\cA$ and $\cBa_{\rm odd}$},  gives an expression of the form~\eqref{inv-tilde-form} for
the series
$$
\sum_{n\ge 2} \tilde{\cD}_{n+2}^{FC}(q)x^{n+1}= \sum_{n\ge 3} \tilde{\cD}_{n+1}^{FC}(q)x^{n},
$$
with the values of  $\cL_0$, $\cR_0$, $\cR, \cR_1, \cR_2$, $\cS$ given in the
theorem. As before, the value of $\cR_3$ only differs from  $\cR_3^D$ by a polynomial in $x$ of degree 2.
\end{proof}


\end{document}